\documentclass[12pt,english]{amsart}

\usepackage{babel}
\usepackage{verbatim}
\usepackage{mathrsfs}
\usepackage{amstext}
\usepackage{amsthm}
\usepackage{amssymb}
\usepackage{amsfonts}
\usepackage{amsmath}
\usepackage{esint}
\usepackage{enumerate}
\usepackage[unicode=true,pdfusetitle,
bookmarks=true,bookmarksnumbered=false,bookmarksopen=false,
breaklinks=false,pdfborder={0 0 1},colorlinks=false]
{hyperref}
\usepackage[margin=2.5cm]{geometry}
\usepackage[foot]{amsaddr}
\usepackage{mathtools}
\usepackage{ dsfont }
\usepackage{color,graphicx}

\numberwithin{equation}{section}
\numberwithin{figure}{section}
\theoremstyle{plain}
\newtheorem{thm}{Theorem}[section]
\theoremstyle{definition}
\newtheorem{rem}[thm]{Remark}
\theoremstyle{definition}
\newtheorem{defn}[thm]{Definition}
\theoremstyle{plain}
\newtheorem{prop}[thm]{Proposition}
\theoremstyle{plain}
\newtheorem{lem}[thm]{Lemma}
\theoremstyle{plain}

\newtheorem*{SSOYconjecture}{SSOY predictability conjecture}
\theoremstyle{plain} 
\newtheorem{cor}[thm]{Corollary}

\theoremstyle{definition}

\theoremstyle{definition}

\theoremstyle{definition}

\theoremstyle{definition}
\newtheorem*{convention}{Convention}

\theoremstyle{definition}
\newtheorem*{Claim1}{Claim 1}
\newtheorem*{Claim2}{Claim 2}
\newtheorem*{Claim3}{Claim 3}
\newtheorem*{Claim4}{Claim 4}

\DeclareMathOperator{\dist}{dist}

\DeclareMathOperator{\Leb}{Leb}

\DeclareMathOperator{\supp}{supp}

\DeclareMathOperator{\Var}{Var}

\newcommand{\R}{\mathbb R}
\newcommand{\Z}{\mathbb Z}
\newcommand{\N}{\mathbb N}
\newcommand{\Q}{\mathbb Q}
\newcommand{\B}{\mathbb B}

\renewcommand{\Re}{\textup{Re}}

\newcommand{\eps}{\varepsilon}

\newcommand{\mH}{\mathcal{H}}

\newcommand{\hdim}{\dim_H}

\newcommand{\uid}{\overline{\idim}}
\newcommand{\lid}{\underline{\idim}}

\DeclareMathOperator{\Lip}{Lip}
\DeclareMathOperator{\Per}{Per}

\DeclareMathOperator{\idim}{ID}

\begin{document}

\title{On the Shroer--Sauer--Ott--Yorke predictability conjecture for time-delay embeddings}

\author[K. Bara\'{n}ski]{Krzysztof Bara\'{n}ski$^1$}
\address{$^1$Institute of Mathematics, University of Warsaw, ul.~Banacha 2, 02-097 Warszawa, Poland}
\email{baranski@mimuw.edu.pl}

\author[Y. Gutman]{Yonatan Gutman$^2$}
\address{$^2$Institute of Mathematics, Polish Academy of Sciences,
ul.~\'Sniadeckich 8, 00-656 Warszawa, Poland}
\email{y.gutman@impan.pl}

\author[A. \'{S}piewak]{Adam \'{S}piewak$^3$}
\address{$^3$Department of Mathematics, Bar-Ilan University, Ramat Gan, 5290002, Israel}
\email{ad.spiewak@gmail.com}

\begin{abstract}
Shroer, Sauer, Ott and Yorke conjectured in 1998 that the Takens delay embedding theorem can be improved in a probabilistic context. More precisely, their conjecture states that if $\mu$ is a natural measure for a smooth diffeomorphism of a Riemannian manifold and $k$ is greater than the information dimension of $\mu$, then $k$ time-delayed measurements of a one-dimensional observable $h$ are generically sufficient for a \emph{predictable} reconstruction of $\mu$-almost every initial point of the original system. This reduces by half the number of required measurements, compared to the standard (deterministic) setup. We prove the conjecture for ergodic measures and  show that it holds for a generic smooth diffeomorphism, if the information dimension is replaced by the Hausdorff one. To this aim, we prove a general version of predictable embedding theorem for injective Lipschitz maps on compact sets and arbitrary Borel probability measures. We also construct an example of a $C^\infty$-smooth diffeomorphism with a natural measure, for which the conjecture does not hold in its original formulation.
\end{abstract}



\maketitle
\section{Introduction}\label{sec:intro}

\subsection{General background}
This paper concerns probabilistic aspects of the \emph{Takens delay embedding theorem}, dealing with the problem of reconstructing a dynamical system from a sequence of measurements of a one-dimensional observable. More precisely, let $T\colon X \to X$ be a transformation on a \emph{phase space} $X$. Fix $k \in \N$ and consider a function (\emph{observable}) $h \colon X \to \R$ together with the corresponding $k$-\emph{delay coordinate map} 
\[
\phi \colon X \to \R^k, \qquad \phi(x) = (h(x), \ldots, h(T^{k-1}x)).
\]
Takens-type delay embedding theorems state that if $k$ is large enough, then $\phi$ is an embedding (i.e.~is injective) for a typical observable $h$. The injectivity of $\phi$ ensures that an (unknown) initial state $x \in X$ of the system can be uniquely recovered from the sequence of $k$ measurements $h(x), \ldots, h(T^{k-1}x)$ of the observable $h$, performed along the orbit of $x$. It also implies that the dynamical system $(X,T)$ has a reliable model in $\R^k$ of the form $(\tilde X, \tilde T) = (\phi(X), \phi \circ T \circ \phi^{-1})$. 

This line of research originates from the seminal paper of Takens \cite{T81} on diffeomorphisms of compact manifolds. Extensions of Takens' work were obtained in several categories, e.g.~in \cite{SYC91, 1999delay, CaballeroEmbed, Rob05, Gut16, GQS18, StarkEmbedSurvey, StarkStochEmbed, NV18} (see also \cite{Rob11, BGS19} for a more detailed overview). A common feature of these results is that the minimal number of measurements sufficient for an exact reconstruction of the system is $k \approx 2\dim X $, where $\dim X $ is the dimension of the phase space $X$. This threshold agrees with the one appearing in the classical non-dynamical embedding theorems (e.g.~Whitney theorem \cite{Whitney36}, Menger--N\"{o}beling theorem \cite[Theorem~V.2]{HW41} and Ma\~{n}\'{e} theorem \cite[Theorem 6.2]{Rob11}). It is worth to notice that Takens-type theorems serve as a justification of the validity of time-delay based procedures, which are actually used in applications (see e.g.~\cite{hgls05distinguishing, KY90,sgm90distinguishing,sm90nonlinear}) and have been met with a great interest among mathematical physicists (see e.g.~\cite{PCFS80, HBS15, SYC91, Voss03}).

In 1998, Shroer, Sauer, Ott and Yorke conjectured (see \cite[Conjecture 1]{SSOY98}), that for smooth diffeomorphisms on compact manifolds, in a probabilistic setting (i.e.~when the initial point $x \in X$ is chosen randomly according to a \emph{natural} probability measure $\mu$), the number of measurements required for an almost sure \emph{predictable} reconstruction of the system can be generically reduced by half, up to the \emph{information dimension} of $\mu$. A precise formulation is given below in Subsection~\ref{subsec:SSOYconj}. We will refer to this conjecture as \emph{Shroer--Sauer--Ott--Yorke predictability conjecture} or \emph{SSOY predictability conjecture}.
In \cite{SSOY98}, the authors provided some heuristic arguments supporting the conjecture together with its numerical verification for some examples (H\'enon and Ikeda maps). However, a rigorous proof of the conjecture has been unknown up to now. 

In this paper, we prove a general version of a predictable embedding theorem (Theorem~\ref{thm:conj_true}), valid for injective Lipschitz transformations of compact sets and arbitrary Borel probability measures, which shows that an almost sure predictable reconstruction of the system is possible with the number of measurements reduced to the \emph{Hausdorff dimension} of $\mu$, under a mild assumption bounding the dimensions of sets of periodic points of low periods. As a corollary, we obtain the SSOY predictability conjecture for generic smooth $C^r$-diffeomorphisms on compact manifolds for $r \geq 1$, with information dimension replaced by the Hausdorff one (Corollary~\ref{cor:conj_true-generic}) and the original conjecture for arbitrary $C^r$-diffeomorphisms and ergodic measures (Corollary~\ref{cor:conj_true-ergodic}). We also construct an example of a $C^\infty$-smooth diffeomorphism of a compact Riemannian manifold with a non-ergodic natural measure, for which the original conjecture does not hold (Theorem~\ref{thm:counterexample_main}). This shows that in a general case, the change of the information dimension to the Hausdorff one is necessary.

Let us note that the SSOY predictability conjecture has been invoked in a number of papers (see e.g.~\cite{Liu10, McSharrySmith04, OrtegaLouis98}) as a theoretical argument for reducing the number of measurements required for a reliable reconstruction of the system, also in applications (see e.g.~\cite{Epilepsy} studying neural brain activity in focal epilepsy). Our result provides a mathematically rigorous proof of the correctness of these procedures. 

\subsection{Shroer--Sauer--Ott--Yorke predictability conjecture}\label{subsec:SSOYconj}
Before we formulate the conjecture stated in \cite{SSOY98} in a precise way, we need to introduce some preliminaries, in particular the notion of predictability. In the sequel, we consider a general situation, when the phase space $X$ is an arbitrary compact set in $\R^N$ (note that by the Whitney embedding theorem \cite{Whitney36}, we can assume that a smooth compact manifold is embedded in $\R^N$ for sufficiently large $N$). We denote the (topological) support of a measure $\mu$ by $\supp \mu$ and write $\phi_*\mu$ for a push-forward of $\mu$ by a measurable transformation $\phi$, defined by $\phi_*\mu(A) = \mu(\phi^{-1}(A))$ for measurable sets $A$.

\begin{defn}\label{def:predictability} Let $X \subset \R^N$ be a compact set, let $\mu$  be a Borel probability measure with support in $X$ and let $T \colon X \to X$ be a Borel transformation (i.e.~such that the preimage of any Borel set is Borel). Fix $k \in \N$. Let $h \colon X \to \R$ be a Borel observable and let $\phi \colon X \to \R^k$ given by $\phi (x) = (h(x), \ldots, h(T^{k-1}x))$ be the corresponding $k$-delay coordinate map. Set $\nu = \phi_*\mu$ (considered as a Borel  measure in $\R^k$) and note that $\supp \nu \subset \phi(X)$. For $y \in \supp \nu$ and $\eps>0$ define
\begin{align*}
\chi_{\eps}(y) &= \frac{1}{\mu\big(\phi^{-1}(B(y, \eps))\big)} \int \limits_{\phi^{-1}(B(y, \eps))} \phi(Tx)d\mu(x),\\
\sigma_{\eps}(y) &= \bigg(\frac{1}{\mu\big(\phi^{-1}(B(y, \eps))\big)} \int \limits_{\phi^{-1}(B(y, \eps))} \|\phi(Tx) -\chi_{\eps}(y)\|^2d\mu(x)\bigg)^{\frac{1}{2}},
\end{align*}
where $B(y,\eps)$ denotes the open ball of radius $\eps$ centered at $y$.
In other words, $\chi_{\eps}(y)$ is the conditional expectation of the random variable $\phi \circ T$ (with respect to $\mu$) given $\phi \in B(y, \eps)$, while $\sigma_{\eps}(y)$ is its conditional standard deviation. Define also the \emph{prediction error} at $y$ as
\[ \sigma(y) = \lim \limits_{\eps \to 0} \sigma_{\eps}(y), \]
provided the limit exists. A point $y$ is said to be \emph{predictable} if $\sigma(y)=0$.
\end{defn}

Note that the prediction error depends on the observable $h$. We simplify the notation by suppressing this dependence. 

\begin{rem}
Note that the predictability of points of the support of the measure $\nu$ does not imply that the delay coordinate map $\phi$ is injective. Indeed, if $h$ (and hence $\phi$) is constant, then every point $y \in \supp \nu$ is predictable.
\end{rem}

\begin{rem}[{\bf Farmer and Sidorowich algorithm}{}] 
As explained in \cite{SSOY98}, the notion of predictability arises naturally in the context of a prediction algorithm proposed by Farmer and Sidorowich in \cite{FarmerSidorowich87}. To describe it, suppose that for a point $x \in X$ we are given a sequence of measurements $h(x), \ldots, h(T^{n+k-1}(x))$ of the observable $h$ for some $n \in \N$. This defines a sequence of $k$-delay coordinate vectors of the form
\[
y_i = (h(T^ix), \ldots, h(T^{i+k-1}x)), \qquad i = 0, \ldots, n.
\]
Knowing the sample values of $y_0, \ldots, y_n$, we would like to predict the one-step future of the model, i.e.~the value of the next point $y_{n+1} = (h(T^{n+1}x), \ldots, h(T^{n+k}x))$. For a small $\eps>0$ we define the predicted value of $y_{n+1}$ as
\[ \widehat{y_{n+1}} = \frac{1}{\#\mathcal I} \sum_{i \in \mathcal I} y_{i+1} \quad \text{for} \quad \mathcal I = \{ 0 \leq i < n : y_i \in B(y_n, \eps) \}. \]
In other words, the predicted value of $y_{n+1}$ is taken to be the average of the values $y_{i+1}$, where we count only those $i$, for which $y_i$ are $\eps$-close to the last known point $y_n$. 

Notice that if the $k$-delay coordinate map $\phi$ is an embedding, then the points $y_i$ form an orbit of $y_0$ under the model transformation $\tilde T$ defined by the delay coordinate map $\phi$, i.e.~$y_i = \tilde T^i(y_0)$ for $(\tilde X, \tilde T) = (\phi(X), \phi\circ T \circ \phi^{-1})$. Hence, in this case the predicted value $y_{n+1} = \tilde T(y_n)$ is the average of the values $y_{i+1} = \tilde T(y_i)$, $i \in \mathcal I$. 

If the initial point $x \in X$ is chosen randomly according to an ergodic probability measure $\mu$, then for $n \to \infty$, the collection of points $y_i,\ i \in \mathcal{I}$ is asymptotically distributed in $B(y_n, \eps)$ according to the measure $\nu = \phi_*\mu$. Therefore, the value of $\sigma_\eps(y_n)$ from Definition~\ref{def:predictability} approaches asymptotically the standard deviation of the predicted point $\widehat{y_{n+1}}$. The condition of predictability states that this  standard deviation converges to zero as $\eps$ tends to zero.
\end{rem}

In \cite{SSOY98}, the Shroer--Sauer--Ott--Yorke predictability conjecture is stated for a special class of measures, called \emph{natural measures}. To define it, recall first that a measure $\mu$ on $X$ is \emph{invariant} for a measurable map $T\colon X \to X$ if $\mu(T^{-1}(A))=\mu(A)$ for every measurable set $A \subset X$. A set $\Lambda \subset X$ is called \emph{$T$-invariant} if $T(\Lambda) \subset \Lambda$.
\begin{defn}
	Let $X$ be a compact Riemannian manifold and $T \colon X \to X$ be a smooth diffeomorphism. A compact $T$-invariant set $\Lambda \subset X$ is called an \emph{attractor}, if the set $B(\Lambda) = \{ x \in X : \lim_{n \to \infty} \dist(T^n x, \Lambda) = 0 \}$ is an open set containing $\Lambda$. The set $B(\Lambda)$ is called the \emph{basin of attraction} to $\Lambda$. A $T$-invariant Borel probability measure $\mu$ on $\Lambda$ is called a \emph{natural measure} if 
	\[ \lim \limits_{n \to \infty} \frac{1}{n} \sum \limits_{i=0}^{n-1} \delta_{T^i x} = \mu \]
	for almost every $x \in B(\Lambda)$ with respect to the volume measure on $X$, where $\delta_y$ denotes the Dirac measure at $y$ and the limit is taken in the weak-$^*$ topology.
\end{defn}

\begin{rem}
Note that in ergodic theory of dynamical systems, some authors use the name \emph{physical measure} or \emph{SRB $($Sinai--Ruelle--Bowen$)$ measure} for similar concepts (see e.g. \cite{Y02}). The term `natural measure' occurs commonly in mathematical physics literature (see e.g.~\cite{Ott02,OttYorke08}).
\end{rem}

\begin{defn}
For a Borel probability measure $\mu$ in $\R^N$ with compact support define its \emph{lower} and \emph{upper information dimensions} as
\footnote{Information dimensions are often defined in an equivalent way as
		\[ \lid(\mu) = \liminf \limits_{\eps \to 0} \frac{1}{\log \eps} \sum \limits_{C \in \mathcal{C}_{\eps}}\mu(C)\log\mu(C), \qquad \uid (\mu) = \limsup \limits_{\eps \to 0} \frac{1}{\log \eps} \sum \limits_{C \in \mathcal{C}_{\eps}}\mu(C)\log\mu(C),\] 
		where $\mathcal{C}_{\eps}$ is the partition of $\R^N$ into cubes with side lengths $\eps$ and vertices in the lattice $(\eps\Z)^N$ (see e.g.~\cite[Appendix I]{WV10}).}
    
\[ \lid(\mu) = \liminf \limits_{\eps \to 0} \int \limits_{\supp \mu} \frac{\log \mu(B(x,\eps))}{\log \eps} d\mu(x), \qquad\uid(\mu) = \limsup \limits_{\eps \to 0} \int \limits_{\supp \mu} \frac{\log \mu(B(x,\eps))}{\log \eps} d\mu(x).\]
If $\lid(\mu) = \uid(\mu)$, then we denote their common value as $\idim (\mu)$ and call it the \emph{information dimension} of $\mu$. 
\end{defn}

We are now ready to state the SSOY predictability conjecture in its original form as stated in \cite{SSOY98}. Recall that for a map $T\colon X \to X$ with a Borel probability measure $\mu$, a number $k \in \N$ and a function $h \colon X \to \R$, we consider  the $k$-delay coordinate map for the observable $h$ defined by
\[
\phi(x) =  \phi_{h,k} (x) = (h(x), \ldots, h(T^{k-1}x)).
\]
To emphasize the dependence on $h$ and $k$, we will write $\phi_{h,k}$ for $\phi$ and $\nu_{h,k}$ for the push-forward measure $\nu = \nu_{h,k} = (\phi_{h,k})_*\mu$. 
 
\begin{SSOYconjecture}[{\cite[Conjecture 1]{SSOY98}}]\label{con:1}
Let $T\colon X \to X$ be a smooth diffeomorphism of a compact Riemannian manifold $X$ and let $\Lambda \subset X$ be an attractor of $T$ with a natural measure $\mu$ such that $\idim(\mu) = D$. Fix $k>D$. Then $\nu_{h,k}$-almost every point of $\R^k$ is predictable for a generic observable $h\colon X \to \R$.
\end{SSOYconjecture}

Note that in this formulation some details (e.g.~the type of genericity and the smoothness class of the dynamics) are not specified precisely. 

\subsection{Main results}

Now we present the main results of the paper. First, we state a predictable embedding theorem, which holds in a general context of injective Lipschitz maps $T$ on a compact set $X \subset \R^N$ equipped with a Borel probability measure $\mu$. Recall that by the Whitney embedding theorem \cite{Whitney36}, we can assume that a smooth compact manifold is embedded in $\R^N$ for sufficiently large $N$. Our observation is that in this generality, the predictability  holds if we replace the information dimension $\idim(\mu)$ by the \emph{Hausdorff dimension} $\dim_H \mu$ (see Subsection~\ref{subsec:dim} for definition).  

In the presented results, we understand the genericity of the observable $h$ in the sense of \emph{prevalence} in the space $\Lip(X)$ of Lipschitz observables $h \colon X \to \R$ (with a polynomial \emph{probe set}), which is an analogue of the `Lebesgue almost sure' condition in infinite dimensional spaces (see Subsection~\ref{subsec:prev} for precise definitions). In particular, the genericity of $h$ holds also in the sense of prevalence in the space of $C^r$-smooth observables $h \colon X \to \R$, for $r \ge 1$. Let us note that it is standard to use prevalence as a notion of genericity in the context of Takens-type embedding theorems (see e.g.~\cite{SYC91, Rob11}).

It is known that Takens-type theorems require some bounds on the size of sets of $T$-periodic points of low periods. Following \cite{BGS19}, we assume $\hdim(\mu|_{\Per_p(T)}) < p$ for $p=1, \ldots, k-1$, where 
\[
\Per_p(T) = \{ x \in X : T^p x= x \}.
\]
With these remarks, our main result is the following. 

\begin{thm}[{\bf Predictable embedding theorem for Lipschitz maps}{}]\label{thm:conj_true}
Let $X \subset \R^N$ be a compact set, let $\mu$ be a Borel probability measure on $X$ and let $T\colon X \to X$ be an injective Lipschitz map. Take $k>\hdim\mu$ and assume $\hdim(\mu|_{\Per_p(T)}) < p$ for $p=1, \ldots, k-1$. Then for a prevalent set of Lipschitz observables $h\colon X \to \R$, the $k$-delay coordinate map $\phi_{h,k}$ is injective on a Borel set of full $\mu$-measure, and $\nu_{h,k}$-almost every point of $\R^k$ is predictable.
\end{thm}

\begin{rem}
Notice that except of predictability, we obtain almost sure injectivity of the delay coordinate map, which means that the system can be reconstructed in $\R^k$ in a one-to-one fashion on a set of full measure. 
\end{rem}

An extended version of Theorem~\ref{thm:conj_true} is proved in Section~\ref{sec:conj} as Theorem~\ref{thm:convergence_takens}.

Note that the assumption on the dimension of $\mu$ restricted to the set of $p$-periodic points can be omitted if there are only finitely many periodic points of given period. By the Kupka--Smale theorem (see \cite[Chapter 3, Theorem 3.6]{palis1982geometric}), the latter condition is generic (in the Baire category sense) in the space of $C^r$-diffeomorphisms, $r\geq 1$, of a compact manifold, equipped with the uniform $C^r$-topology (see \cite{BGS19} for more details). Therefore, we immediately obtain the SSOY predictability conjecture for generic smooth $C^r$-diffeomorphisms, with information dimension replaced by the Hausdorff one.

\begin{cor}[{\bf SSOY predictability conjecture for generic diffeomorphisms}{}]\label{cor:conj_true-generic}
Let $X$ be a compact Riemannian manifold and $r \ge 1$. Then for a $C^r$-generic  diffeomorphism $T\colon X \to X$ with a natural measure $\mu$ $($or, more generally, any Borel probability measure$)$ and $k > \dim_H \mu$, for a prevalent set $($depending on $T)$ of Lipschitz observables $h\colon X \to \R$, the $k$-delay coordinate map $\phi_{h,k}$ is injective on a set of full $\mu$-measure, and $\nu_{h,k}$-almost every point of $\R^k$ is predictable.
\end{cor}

Suppose now the measure $\mu$ in Theorem~\ref{thm:conj_true} is $T$-invariant and ergodic. Then we have $\hdim\mu \leq \lid(\mu) \leq \uid(\mu)$ (see Proposition~\ref{prop:ergodic_exact_dim}). Moreover, either the set of $T$-periodic points has $\mu$-measure zero, or $\mu$ is supported on a periodic orbit of $T$ (see the proof of \cite[Remark~4.4(c)]{BGS19}. Hence, the assumption on the dimension of $\mu$ restricted to the set of $p$-periodic points can again be omitted. This proves the original SSOY conjecture for arbitrary $C^r$-diffeomorphisms and ergodic measures.

\begin{cor}[{\bf SSOY predictability conjecture for ergodic measures}{}]\label{cor:conj_true-ergodic}
Let $X$ be a compact Riemannian manifold, $r \ge 1$, and let $T\colon X \to X$ be a $C^r$-diffeomorphism with an ergodic natural measure $\mu$ $($or, more generally, any $T$-invariant ergodic Borel probability measure$)$. Take $k>\lid(\mu)$. Then for a prevalent set of Lipschitz observables $h\colon X \to \R$, the $k$-delay coordinate map $\phi_{h,k}$ is injective on a set of full $\mu$-measure, and $\nu_{h,k}$-almost every point of $\R^k$ is predictable.
\end{cor}

Our final result is that the SSOY predictability conjecture does not hold in its original formulation for all smooth diffeomorphisms, i.e.~the condition $k>\idim(\mu)$ is not sufficient for almost sure predictability for generic observables, even if $\mu$ is within the class of natural measures.

\begin{thm}\label{thm:counterexample_main}
There exists a $C^{\infty}$-smooth diffeomorphism of the $3$-dimensional compact Riemannian manifold $X = \mathbb{S}^2 \times \mathbb{S}^1$ with a natural measure $\mu$, such that $\idim(\mu)<1$ and for a prevalent set of Lipschitz observables $h\colon X \to \R$, there exists a positive $\nu_{h,1}$-measure set of non-predictable points. In particular, the set of Lipschitz observables $h\colon X \to \R$ for which 
$\nu_{h,1}$-almost every point of $\R^k$ is predictable, is not prevalent. 
\end{thm}

The construction is presented in Section~\ref{sec:counterexample} (see Theorem~\ref{thm:natural_counterexample} for details).

\begin{rem}
Theorem \ref{thm:counterexample_main} shows that the original SSOY predictability conjecture fails for a specific system $(X,T)$. It remains an open question whether it holds for a generic $C^r$-diffeomorphism $T$ of a given compact Riemannian manifold $X$. By Corollary~\ref{cor:conj_true-generic}, this would follow from the dimension conjecture of Farmer, Ott and Yorke \cite[Conjecture~1]{FOY83}, which (in particular) states that the Hausdorff and information dimension of the natural measure typically coincide.
\end{rem}

\subsection*{Organization of the paper}

Section~\ref{sec:prelim} contains preliminary material, gathering definitions and tools required for the rest of the paper. Theorem~\ref{thm:conj_true} and its extension Theorem~\ref{thm:convergence_takens} are proved in Section~\ref{sec:conj}. Section~\ref{sec:counterexample} contains a construction of the example presented in Theorem~\ref{thm:counterexample_main}, divided into several steps.

\subsection*{Acknowledgements}
We are grateful to Edward Ott for bringing the paper \cite{SSOY98} to our attention and to Bal\'azs B\'ar\'any for informing us about the results of \cite{SimmonsRohlin}. We also thank K\'aroly Simon for useful discussions. KB and A\'S were partially supported by the National Science Centre (Poland) grant 2019/33/N/ST1/01882. YG was partially supported by the National Science Centre (Poland) grant 2020/39/B/ST1/02329.

\section{Preliminaries}\label{sec:prelim}

\subsection{Hausdorff and information dimensions}\label{subsec:dim}
For $s>0$, the \emph{$s$-dimensional $($outer$)$ Hausdorff measure} of a set $X \subset \R^N$ is defined  as
	\[ \mH^s(X) = \lim \limits_{\delta \to 0}\ \inf \Big\{ \sum \limits_{i = 1}^{\infty} |U_i|^s : X \subset \bigcup \limits_{i=1}^{\infty} U_i,\ |U_i| \leq \delta  \Big\},\]
	where $| \cdot |$ denotes the diameter of a set (with respect to the Euclidean distance in $\R^N)$.
	The \emph{Hausdorff dimension} of $X$ is given as
	\[ \hdim X = \inf \{ s > 0 : \mathcal{H}^s(X) = 0 \} = \sup \{ s > 0 : \mathcal{H}^s(X) = \infty \}. \]
	The (upper) Hausdorff dimension of a finite Borel measure $\mu$ in $\R^N$ is defined as
	\[
	\hdim \mu = \inf\{ \hdim X: X \subset \R^N \text{ is a Borel set of full $\mu$-measure} \}.
	\]
	By the Whitney embedding theorem \cite{Whitney36}, we can assume that a smooth compact manifold is smoothly embedded in the Euclidean space, hence the Hausdorff dimension is well defined also for Borel measures on manifolds.

In general, $\lid(\mu)$ and $\uid(\mu)$ are not comparable with $\hdim\mu$ (see \cite[Section 3]{FanLauRao02}). One can however obtain inequalities between them for measures which are ergodic with respect to Lipschitz transformations. 

\begin{prop}\label{prop:ergodic_exact_dim} Let $X \subset \R^N$ be a closed set, let $T \colon X \to X$ be a Lipschitz map and let $\mu$ be a $T$-invariant and ergodic Borel probability measure on $X$. Then 
\[
\hdim\mu \leq \lid(\mu) \leq \uid(\mu).
\]
\end{prop}

\begin{proof}
The inequality $\lid(\mu) \leq \uid(\mu)$ is obvious. The estimate $\hdim\mu \leq \lid(\mu)$ follows by combining \cite[Propositions~10.2--10.3]{FalconerTechniques} with \cite[Theorem~1.3]{FanLauRao02} and \cite[Proposition~10.6]{FalconerTechniques}.
\end{proof}

For more information on dimension theory in Euclidean spaces we refer to \cite{falconer2004fractal, mattila, Rob11}.

\subsection{Prevalence}\label{subsec:prev}

In the formulation of our results, the genericity of the considered observables is understood in terms of \emph{prevalence} -- a notion introduced by Hunt, Shroer and Yorke in \cite{Prevalence92}, which is regarded to be an analogue of `Lebesgue almost sure' condition in infinite dimensional normed linear spaces.
	
	\begin{defn}
		Let $V$ be a normed space. A Borel set $S \subset V$ is called \emph{prevalent} if there exists a Borel measure $\nu$ in $V$, which is positive and finite on some compact set in $V$, such that for every $v \in V$, the vector $v + e$ belongs to $S$ for $\nu$-almost every $e \in V$. A non-Borel subset of $V$ is prevalent if it contains a prevalent Borel subset.
	\end{defn}

We will apply this definition to the space $\Lip(X)$ of all Lipschitz functions $h \colon X \to \R$ on a compact metric space $X$, endowed with the Lipschitz norm $\|h\|_{\Lip} = \|h\|_{\infty} + \Lip(h)$, where $\|h\|_\infty$ is the supremum norm and $\Lip(h)$ is the Lipschitz constant of $h$. We will use the following standard condition, which is sufficient for prevalence. Let $\{h_1, \ldots, h_m\}$, $m \in \N$, be a finite set of functions in $\Lip(X)$, called the \emph{probe set}. Define $\xi \colon \R^m \to \Lip(X)$ by $\xi(\alpha_1, \ldots, \alpha_m) = \sum_{j=1}^m \alpha_j h_j$. Then $\nu = \xi_*\Leb$, where $\Leb$ is the Lebesgue measure in $\R^k$, is a Borel measure in $\Lip(X)$, which is positive and finite on the compact set $\xi([0,1]^m)$. For this measure, the sufficient condition for a set $S \subset \Lip(X)$ to be prevalent is that for every $h \in \Lip(X)$, the function $h + \sum_{j=1}^m \alpha_j h_j$ is in $S$ for Lebesgue almost every $(\alpha_1, \ldots, \alpha_m) \in \R^m$. In this case, we say that $S$ is \emph{prevalent} in $\Lip(X)$ \emph{with the probe set} $\{h_1, \ldots, h_m\}$. 

For more information on prevalence we refer to \cite{Prevalence92} and \cite[Chapter~5]{Rob11}.

\subsection{Probabilistic Takens delay embedding theorem}

To prove Theorem~\ref{thm:conj_true}, we will use our previous result from \cite{BGS19}, which we recall below, using the notion of prevalence described in Subsection~\ref{subsec:prev}. This is a probabilistic version of the Takens delay embedding theorem, asserting that under suitable conditions on $k$, there is a prevalent set of Lipschitz observables, which give rise to an almost surely injective $k$-delay coordinate map. 
 
\begin{thm}[{\bf Probabilistic Takens delay embedding theorem}, {\cite[Theorem~4.3 and Remark~4.4]{BGS19}}]\label{thm:takens}
	Let $X \subset \R^N$ be a compact set, $\mu$ a Borel probability measure on $X$ and $T\colon X \to X$ an injective Lipschitz map. Take $k > \hdim\mu$ and assume $\hdim(\mu|_{\Per_p(T)}) < p$ for $p=1, \ldots, k-1$. Let $S$ be the set of Lipschitz observables $h \colon X \to \R$, for which the $k$-delay coordinate map $\phi_{h,k}$ is injective on a Borel set $X_h \subset X$ of full $\mu$-measure. Then $S$ is prevalent in $\Lip(X)$ with the probe set equal to a linear basis of the space of real polynomials of $N$ variables of degree at most $2k-1$. If $\mu$ is additionally $T$-invariant, then the set $X_h$ for $h \in S$ can be chosen to satisfy $T(X_h) = X_h$. 
\end{thm}

\subsection{Topological Rokhlin disintegration theorem}

A useful tool connecting the probabilistic Takens delay embedding theorem and the SSOY predictability conjecture is the following topological version of the Rokhlin disintegration theorem in compact metric spaces. The Rokhlin disintegration theorem (see e.g.~\cite{Rohlin52}) is a classical result on the existence and almost sure uniqueness of the system of conditional measures. The crucial fact for us is that in the topological setting, the conditional measures can be defined as limits of conditional measures on preimages of shrinking balls, where the convergence holds almost surely, as was proved by Simmons in \cite{SimmonsRohlin}. 

In the context of the Rokhlin disintegration theorem, one assumes that the considered   measures are \emph{complete}, i.e.~every subset of a zero-measure set is measurable. Recall that every finite Borel measure $\mu$ on a metric space $X$ has an extension (\emph{completion}) to a complete measure on the $\sigma$-algebra of $\mu$-measurable sets, i.e. the smallest $\sigma$-algebra containing all Borel sets in $X$ and all subsets of zero $\mu$-measure Borel sets. In other words, every $\mu$-measurable set $A$ can be expressed as $A = B \cup C$, where $B$ is a Borel set and $C \subset D$ for some Borel set $D$ with $\mu(D) = 0$ (see e.g.~\cite[Theorem 1.19]{FollandAnalysis} for the case $X = \R$). Alternatively, this $\sigma$-algebra is obtained as a family of sets measurable with respect to the outer measure generated by $\mu$ (see e.g.~\cite[Example~22, p.~32]{FollandAnalysis}). Recall also that a function $\psi\colon X \to \R$ is called $\mu$-measurable if $\psi^{-1}(B)$ is $\mu$-measurable for every Borel set $B \subset \R$.

\begin{defn}\label{defn:cond_measures}  Let $X$ be a compact metric space and let $\mu$ be a complete Borel probability measure on $X$. Let $Y$ be a separable Riemannian manifold and let $\phi \colon X \to Y$ be a Borel map. Set $\nu = \phi_* \mu$ (considered as a complete Borel measure in $Y$). A family $\{ \mu_y : y \in Y\}$ is a \emph{system of conditional measures} of $\mu$ with respect to $\phi$, if
	\begin{enumerate}[$(1)$]
		\item\label{item:cond_supp} for every $y \in Y,\ \mu_y$ is a $($possibly zero$)$ Borel measure on $\phi^{-1}(\{y\})$,
		\item for $\nu$-almost every $y \in Y$, $\mu_y$ is a Borel probability measure,
		\item\label{item:cond_total_prob} for every $\mu$-measurable set $A \subset X$, the function $Y \ni y \mapsto \mu_y(A)$ is $\nu$-measurable and
		\[ \mu(A) = \int \limits_{Y} \mu_y(A)d\nu(y). \]
		
	\end{enumerate}
	
We say that system of conditional measures $\{ \mu_y : y \in Y\}$ is \emph{unique}, if  for every family $\{ \tilde\mu_y : y \in Y \}$ satisfying \eqref{item:cond_supp}--\eqref{item:cond_total_prob}, we have $\tilde\mu_y = \mu_y$ for $\nu$-almost every $y \in Y$.
\end{defn}

\begin{thm}[{\bf Topological Rokhlin disintegration theorem}, {\cite[Theorems 2.1--2.2]{SimmonsRohlin}}]\label{thm:top_rohlin}  Let $X$ be a compact metric space and let $\mu$ be a Borel probability measure on $X$. Let $Y$ be a separable Riemannian manifold and let $\phi \colon X \to Y$ be a Borel map. Set $\nu = \phi_* \mu$. Then for $\nu$-almost every $y \in \supp \nu$ and $\eps > 0$, the conditional probability measures 
\[
\mu_{y,\eps} = \frac{1}{\mu(\phi^{-1}(B(y,\eps)))} \mu|_{\phi^{-1}(B(y,\eps))}
\]
converge in weak-$^*$ topology to a Borel probability measure $\mu_y$ as $\eps$ tends to $0$. Moreover, the collection of measures $\{ \mu_y : y \in Y\}$, where we set $\mu_y = 0$ if $y \notin \supp \nu$ or the convergence does not hold, is a unique system of conditional measures of $\mu$ with respect to $\phi$. 
\end{thm}

The proof of the above theorem is based on the differentiation theorem for finite Borel measures, see \cite[Theorem 9.1]{SimmonsRohlin} for details.

\section{Proof of the predictable embedding theorem for Lipschitz maps}\label{sec:conj}

In this section we prove the following extended version of Theorem~\ref{thm:conj_true}, which at the same time is an extension of Theorem~\ref{thm:takens} asserting prevalent almost sure predictability.

\begin{thm}[{\bf Predictable embedding theorem for Lipschitz maps -- extended version}{}]\label{thm:convergence_takens}
Let $X \subset \R^N$ be a compact set, let $\mu$ be a Borel probability measure on $X$ and let $T\colon X \to X$ be an injective and Lipschitz map. Take $k > \hdim\mu$ and assume $\hdim(\mu|_{\Per_p(T)}) < p$ for $p=1, \ldots, k-1$. Then there is a set $S$ of Lipschitz observables $h \colon X \to \R$, such that $S$ is prevalent in $\Lip(X)$ with the probe set equal to a linear basis of the space of real polynomials of $N$ variables of degree at most $2k-1$, and for every $h \in S$, the following assertions hold.
\begin{enumerate}[$(a)$]
\item There exists a Borel set $X_h\subset X$ of full $\mu$-measure, such that the $k$-delay coordinate map $\phi_{h,k}$ is injective on $X_h$.
\item For every $x \in X_h$, $\lim \limits_{\eps \to 0} \mu_{\phi_{h,k}(x), \eps} = \delta_x$ in the weak-$^*$ topology, where $\delta_x$ denotes the Dirac measure at the point $x$.
\item $\nu_{h,k}$-almost every point of $\R^k$ is predictable.
\end{enumerate}
If $\mu$ is additionally $T$-invariant, then the set $X_h$ for $h \in S$ can be chosen to satisfy $T(X_h) = X_h$. 
\end{thm}

The main ingredients of the proof of Theorem~\ref{thm:convergence_takens} are Theorems \ref{thm:takens} and \ref{thm:top_rohlin}. First, notice that under the assumptions of Theorem~\ref{thm:convergence_takens}, we can use Theorem~\ref{thm:top_rohlin} to show  the existence of a system $\{ \mu_y : y \in \R^k\}$ of conditional measures of $\mu$ with respect to $\phi_{h,k}$, such that for $\nu_{h,k}$-almost every $y \in \R^k$, $\mu_y$ is a Borel probability measure in $X$ satisfying
\begin{equation}\label{eq:mu_eps}
\mu_y = \lim_{\eps \to 0} \mu_{y,\eps}
\end{equation}
in weak-$^*$ topology, where 
\[
\mu_{y,\eps} = \frac{1}{\mu(\phi^{-1}_{h,k}(B(y,\eps)))} \mu|_{\phi^{-1}_{h,k}(B(y,\eps))}
\]
for $\eps > 0$. 

The following lemma shows that for $\nu_{h,k}$-almost every $y \in \R^k$, the prediction error $\sigma(y)$ from Definition~\ref{def:predictability} is equal to the standard deviation of the random variable $\phi_{h,k} \circ T$ with respect to the measure $\mu_y$. Note that the lemma is valid for any continuous (non-necessary Lipschitz) maps $T$ and $h$.

\begin{lem}\label{lem:conv_predict} For $\nu_{h,k}$-almost every $y \in \R^k$,
\[
\sigma(y) = \sqrt{\Var_{\mu_y} (\phi_{h,k} \circ T)},
\]
where 
\[
\Var_{\mu_y} ( \phi_{h,k} \circ T ) = \int \limits_{X} \Big\|\phi_{h,k} \circ T - \int \limits_{X}\phi_{h,k} \circ T d \mu_y\Big\|^2 d\mu_y.
\]
\end{lem}

\begin{proof} For simplicity, let us write $\phi = \phi_{h,k}$. Observe first that for $\nu_{h,k}$-almost every $y \in \R^k$, by \eqref{eq:mu_eps} and the continuity of $\phi \circ T$, we have
\begin{equation}\label{eq:chi_conv}
\chi_\eps(y) = \int \limits_X \phi \circ T d\mu_{y,\eps} \underset{\eps \to 0}{\longrightarrow} \chi(y)
\end{equation}
for
\[
\chi(y) = \int \limits_X \phi \circ T d\mu_y.
\]
Moreover,
\begin{align*}
\sigma^2_\eps(y) - \Var_{\mu_y} ( \phi \circ T)
&= \int \limits_X \| \phi \circ T - \chi_\eps(y) \|^2 d\mu_{y,\eps} - \int \limits_{X} \|\phi \circ T - \chi(y)\|^2 d\mu_y\\ 
&= \int \limits_X \| \phi \circ T - \chi_\eps(y) \|^2 d\mu_{y,\eps} - \int \limits_X \| \phi \circ T - \chi(y) \|^2d\mu_{y,\eps}\\
&+ \int \limits_X \| \phi \circ T - \chi(y) \|^2d\mu_{y,\eps} - \int \limits_{X} \|\phi \circ T - \chi(y)\|^2 d\mu_y,\\
&=I + \textit{II}.
\end{align*}
Again by the continuity of $\phi \circ T$, we have $\textit{II} \underset{\eps \to 0}{\longrightarrow} 0$. Furthermore,
\begin{align*}
|I| &\leq
\int \limits_X \big| \| \phi \circ T - \chi_\eps(y) \|^2 - \| \phi \circ T - \chi(y) \|^2 \big|d\mu_{y,\eps}\\
&= 
\int \limits_X \big( \| \phi \circ T - \chi_\eps(y_0) \| + \| \phi \circ T - \chi(y) \| \big) \, \big| \| \phi \circ T - \chi_\eps(y) \| - \| \phi \circ T - \chi(y)\| \big|d\mu_{y,\eps}\\
&\leq 4 \|\phi \circ T\|_\infty \int \limits_X \|\chi_{\eps}(y) - \chi(y)\|d\mu_{y,\eps} = 4 \|\phi \circ T\|_\infty\, \|\chi_{\eps}(y) - \chi(y)\|,
\end{align*}
by the triangle inequality and the fact $\chi_\eps(y) \leq \| \phi \circ T \|_\infty$.
The latter quantity converges to zero by \eqref{eq:chi_conv}. Therefore, $\sigma^2_\eps(y)$ tends to $\Var_{\mu_y} ( \phi \circ T)$ as $\eps \to 0$, so 
$\sigma(y) = \sqrt{\Var_{\mu_y} (\phi \circ T)}$.
\end{proof}

The following corollary is immediate.

\begin{cor}\label{cor:conv_predict} For $\nu_{h,k}$-almost every $y \in \R^k$, $y$ is predictable if and only if $\phi_{h,k} \circ T$ is constant $\mu_y$-almost surely. In particular, $y$ is predictable provided $\mu_y = \delta_x$ for some $x \in X$.
\end{cor}

By Corollary~\ref{cor:conv_predict}, in order to establish almost sure predictability, it is enough to prove the convergence $\lim_{\eps \to 0} \mu_{\phi_{h,k}(x), \eps} = \delta_x$ for almost every $x \in X$. The idea of the proof of Theorem~\ref{thm:convergence_takens} is the following. Theorem~\ref{thm:takens} guarantees that for a prevalent set of observables, the corresponding delay-coordinate map is injective on a set of full $\mu$-measure. On the other hand, Theorem~\ref{thm:top_rohlin} assures that the measures $\mu_{\phi(x),\eps}$ are almost surely convergent as $\eps \to 0$, and the limits form a system of conditional measures of $\mu$ with respect to $\phi_{h,k}$. Almost sure injectivity implies that these conditional measures are almost surely Dirac measures, hence indeed $\lim_{\eps \to 0} \mu_{\phi_{h,k}(x), \eps} = \delta_x$. A detailed proof is presented below.

\begin{proof}[Proof of Theorem \rm\ref{thm:convergence_takens}]
By Theorem \ref{thm:takens}, there exists a prevalent set $S$ of Lipschitz observables $h$, such that for each $h \in S$, the $k$-delay coordinate map $\phi_{h,k}$ is injective on a  Borel set $\tilde X_h \subset X$ of full $\mu$-measure. For $h \in S$, let us denote for simplicity  $\phi = \phi_{h,k}$ and
\[
\tilde Y_h = \phi(\tilde X_h).
\]
Note that $\tilde Y_h$ has full $\nu_{h,k}$-measure. Moreover, $\tilde Y_h$ is Borel, as a continuous and injective image of a Borel set, see \cite[Theorem~15.1]{K95}. Since $\phi$ is injective on $\tilde X_h$, for every $y \in \tilde Y_h$ there exists a unique point $x_y \in \tilde X_h$, such that $\phi(x_y) = y$. For $y \in \R^k$ define
\[
\tilde \mu_y =
\begin{cases}
\delta_{x_y} &\text{for } y \in \tilde Y_h\\
0 &\text{for } y \in \R^k \setminus \tilde Y_h
\end{cases}.
\]
We check that the collection $\{ \tilde\mu_y : y \in \R^k\}$ satisfies the conditions (1)--(3) of Definition~\ref{defn:cond_measures}. The first two conditions are obvious.
To check the third one, take a $\mu$-measurable set $A \subset X$ and note that for $y \in \phi(A \cap \tilde X_h)$, we have $y \in \tilde Y_h$ and $x_y \in A$, so $\tilde \mu_y(A) = \delta_{x_y}(A) = 1$. On the other hand, if $y \in \tilde Y_h \setminus \phi(A \cap \tilde X_h)$, then $x_y \notin A$, so $\tilde \mu_y(A) = \delta_{x_y}(A) = 0$. Since $\tilde \mu_y(A) = 0$ for $y \in \R^k \setminus \tilde Y_h$, we conclude that for
\[
\psi\colon\R^k \to \R, \qquad  \psi(y)= \tilde \mu_y(A)
\]
we have
\begin{equation}\label{eq:delta_meas}
\psi = \mathds{1}_{\phi(A \cap \tilde X_h)}.
\end{equation}
Hence, to show the $\nu_{h,k}$-measurability of $\psi$, it is enough to check that the set $\phi(A \cap \tilde X_h)$ is $\nu_{h,k}$-measurable. To do it, note that since $A$ is $\mu$-measurable, we have $A = B \cup C$, where $B$ is a Borel set and $C \subset D$ for some Borel set $D$ with $\mu(D) = 0$. Hence,  
$\phi(A \cap \tilde X_h) = \phi(B \cap \tilde X_h) \cup \phi(C \cap \tilde X_h)$. The set $\phi(B \cap \tilde X_h)$ is Borel, which again follows from \cite[Theorem~15.1]{K95}, as $\phi$ is continuous and injective on the Borel set $B \cap \tilde X_h$. Similarly, the set $\phi(C \cap \tilde X_h)$ is contained in the Borel set $\phi(D \cap \tilde X_h)$. Since $\tilde X_h$ has full $\mu$-measure, we have 
\[
\nu_{h,k}(\phi(D \cap \tilde X_h)) = \mu(\phi^{-1}(\phi(D \cap \tilde X_h))) = \mu(\phi^{-1}(\phi(D \cap \tilde X_h))\cap \tilde X_h) = \mu(D) = 0. 
\]
This yields the $\nu_{h,k}$-measurability of the set $\phi(A \cap \tilde X_h)$ and the function $\psi$. Moreover, by \eqref{eq:delta_meas},
\begin{align*}
\int \limits_{Y} \tilde \mu_y(A)d\nu_{h,k}(y) &= \nu_{h,k}(\phi(A \cap \tilde X_h))\\
&= \mu(\phi^{-1}(\phi(A \cap \tilde X_h)))\\ 
&= \mu(\phi^{-1}(\phi(A \cap \tilde X_h))\cap \tilde X_h) = \mu(A).
\end{align*}
It follows that $\{ \tilde\mu_y : y \in \R^k\}$ is a system of conditional measures of $\mu$ with respect to $\phi$, so by the uniqueness in Theorem~\ref{thm:top_rohlin} and \eqref{eq:mu_eps}, 
\[
\tilde\mu_y = \mu_y = \lim_{\eps \to 0} \mu_{y,\eps}
\]
for $\nu_{h,k}$-almost every $y \in \R^k$. Since $\tilde Y_h$ is a Borel set of full $\nu_{h,k}$-measure, we have
\begin{equation}\label{eq:Y_h}
\mu_y = \lim_{\eps \to 0} \mu_{y,\eps} = \delta_{x_y}
\end{equation}
for every $y \in Y_h$, where $Y_h \subset \tilde Y_h$ and $Y_h$ is a Borel set of full $\nu_{h,k}$-measure. By Corollary~\ref{cor:conv_predict}, this implies that $\nu_{h,k}$-almost every $y \in \R^k$ is predictable, which proves the assertion (c) in Theorem~\ref{thm:convergence_takens}. 

Define
\[
X_h = \phi^{-1}(Y_h) \cap \tilde X_h.
\]
Then $X_h$ is a Borel full $\mu$-measure subset of $X$. Since $\phi(X_h) \subset Y_h \subset \tilde Y_h$, by \eqref{eq:Y_h} we have
\[
\mu_{\phi(x)} = \lim_{\eps \to 0} \mu_{\phi(x),\eps} = \delta_{x_{\phi(x)}} = \delta_x
\]
for every $x \in X_h$, which shows the assertion (b). Finally, the assertion (a) follows from the fact $X_h \subset \tilde X_h$.

To end the proof of Theorem~\ref{thm:convergence_takens}, note that if the measure $\mu$ is $T$-invariant, we can define $X_h' = \bigcap_{n\in \Z} T^n(X_h)$ to obtain a full $\mu$-measure subset of $X_h$ with $T(X_h') = X_h'$. For details, see the proof of \cite[Remark~4.4(b)]{BGS19}.
\end{proof}

\begin{rem} Similarly as in \cite{BGS19}, the assumptions $\hdim(\mu) < k$ and $\hdim(\mu|_{\Per_p(T)}) < p$ of Theorem \ref{thm:convergence_takens} can be weakened to $\mu \perp \mH^k$ and $\mu|_{\Per_p(T)} \perp \mH^p$, respectively. Moreover, one can prove a version of Theorem~\ref{thm:convergence_takens} for $\beta$-H\"older observables $h\colon X \to \R$, $\beta \in (0,1]$. It is enough to take $k$ with $\mH^{\beta k}(X) = 0$ and assume that $\mu|_{\Per_p(T)}$ is singular with respect to $\mH^{\beta p}$ for $p = 1, \ldots, k-1$, where $\mH^s$ is the $s$-Hausdorff measure. For a precise formulation of required assumptions see \cite[Theorem 4.3]{BGS19}. As previously, the assumption on periodic points can be omitted if the measure $\mu$ is $T$-invariant and ergodic (see \cite[Remark 4.4(c)]{BGS19} and its proof).
\end{rem}

\section{Counterexample to SSOY predictability conjecture -- proof of  Theorem~\ref{thm:counterexample_main}}\label{sec:counterexample}

In this section we prove Theorem~\ref{thm:counterexample_main}, constructing an example of 
a $C^{\infty}$-smooth diffeomorphism $T$ of a compact Riemannian manifold $X$ with an attractor $\Lambda$ endowed with a natural measure $\mu$, such that $\idim(\mu)<1$ 
and for a prevalent set of Lipschitz observables, there is a positive $\nu_{h,1}$-measure set of non-predictable points. In particular, the set of Lip\-schitz observables $h\colon X \to \R$, for which $\nu_{h,1}$-almost sure predictability holds for the $1$-delay coordinate map $\phi_{h,1}$, is not prevalent. Since the proof is quite involved, we shortly describe the subsequent steps.

In Subsection~\ref{subsec:S^1} we construct a model for the natural measure $\mu$. First, we prove that for an irrational rotation on a circle $\mathbb{S}^1 \subset \R^N$ endowed with the Lebesgue measure $\Leb_{\mathbb S^1}$, the only Lipschitz observables $h \colon \mathbb{S}^1\to \R$ such that the almost sure predictability holds for the $1$-delay coordinate map $\phi$, are the constant functions. Then we construct a model $\mu_0$ for the natural measure $\mu$, taking $X_0 = \{p_0\} \cup \mathbb S^1 \subset \R^N$ for some $p_0 \notin \mathbb S^1$ and defining $T_0 \colon X_0 \to X_0$ as the identity on $\{p_0\}$ and an irrational rotation on $\mathbb S^1$. Then the measure $\mu_0 = \delta_{p_0}/2 + \Leb_{\mathbb S^1}/2$ satisfies $\idim(\mu_0) = 1/2 < 1$, yet the only Lipschitz observables $h \colon X_0 \to \R$ yielding almost sure predictability for the $1$-delay coordinate maps are the functions constant on $\mathbb S^1$. The same holds for any  extension $(X,\mu,T)$ of $(X_0,\mu_0, T_0)$ with $X_0 \subset X,\ T|_{X_0} = T_0$ and $\mu = \mu_0$. In particular, the set 
of Lipschitz observables $h\colon X \to \R$ with almost sure predictability for the $1$-delay coordinate map, is not prevalent. Moreover, for a prevalent set of Lipschitz observables, the almost sure predictability does not hold (Corollary~\ref{cor:unnatural_counterexample}).

The main step, performed in Subsections~\ref{subsec:S^2}--\ref{subsec:S^2xS^1} is to realize the model measure $\mu_0$ as a natural measure $\mu$ for a smooth diffeomorphism $T$ of a compact Riemannian manifold $X$. 
In Subsection~\ref{subsec:S^2}, we construct a $C^\infty$-diffeomorphism $f$ of the $2$-dimensional sphere $\mathbb S^2 = \R^2 \cup \{\infty\}$, such that the trajectories of Lebesgue-almost all points of $\mathbb S^2$ spiral towards the invariant unit circle $S = \{(x,y): x^2 + y^2 = 1\}$, spending most of the time in small neighbourhoods of two fixed points $p, q \in S$ (Proposition~\ref{prop:N_i}). It follows that the average of the Dirac measures at $p$ and $q$ is a natural measure for $f$, with the attractor $S$ and basin $\mathbb S^2 \setminus \{(0,0), \infty\}$ (Corollary~\ref{cor:N_i}). Then, in Subsection~\ref{subsec:S^2xS^1}, we take 
\[
X = \mathbb{S}^2 \times \mathbb{S}^1
\]
and define a $C^\infty$-diffeomorphism $T\colon X \to X$ as a skew product of the form
\[
T(z,t) = (f(z), h_{z}(t)), \qquad z \in \mathbb{S}^2,\; t \in \mathbb{S}^1, 
\]
where $h_z$ are diffeomorphisms of $\mathbb{S}^1$ depending smoothly on $z \in \mathbb{S}^2$, such that for $z$ in a neighbourhood of $p$, the map $h_z$ is equal to a map $g\colon\mathbb{S}^1\to\mathbb{S}^1$ with a unique fixed point $0 \in \R/\Z \simeq \mathbb{S}^1$ attracting all points of $\mathbb{S}^1$, while for $z$ in a neighbourhood of $q$, the map $h_z$ is an irrational rotation on $\mathbb{S}^1$. See Figure~\ref{fig:counterexample} for a schematic view of the map $T$. 

\begin{figure}[ht!]
\begin{center}
\includegraphics[height=8cm]{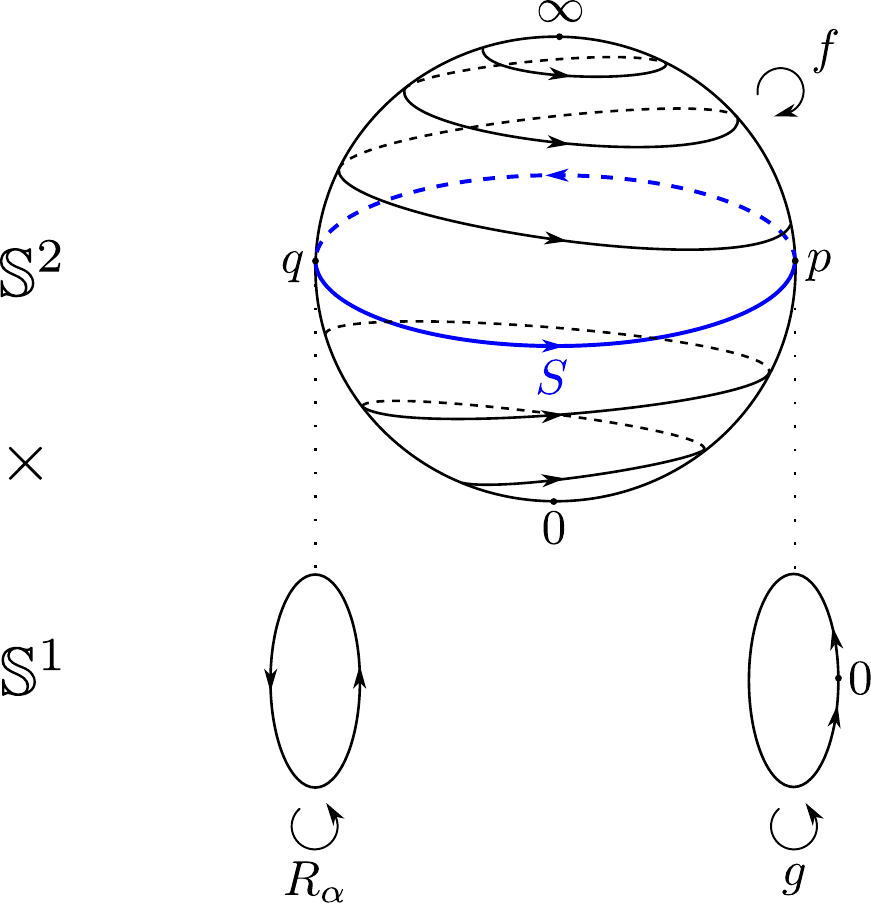}
\end{center}
\caption{Schematic view of the map $T\colon \mathbb S^2 \times \mathbb S^1 \to \mathbb S^2 \times \mathbb S^1$.}
\label{fig:counterexample}
\end{figure}

The map $T$ has an attractor 
\[
\Lambda = S \times \mathbb{S}^1
\]
with the basin $B(\Lambda) = (\mathbb S^2 \setminus\{0, \infty\}) \times \mathbb{S}^1$ and natural measure 
\[
\mu = \frac 1 2 \delta_{p_0} + \frac 1 2 \Leb_{\mathbb S^1},
\]
where $p_0 = (p,0)$ and $\Leb_{\mathbb S^1}$ is the Lebesgue measure on the circle $\{q\} \times \mathbb S^1$ (Theorem~\ref{thm:natural_counterexample}). 
Since the measure $\mu$ is equal to the model measure $\mu_0$, the conclusion follows.

\subsection{Model measure}\label{subsec:S^1}

Consider a circle $\mathbb{S}^1 \subset \R^N$ (by a circle we mean an image of $\{(x,y) \in \R^2: x^2+y^2 = 1\}$ by an affine similarity transformation) with the normalized Lebesgue ($1$-Hausdorff) measure $\Leb_{\mathbb S^1}$ and a rotation $R_\alpha \colon \mathbb{S}^1\to \mathbb{S}^1$ by an angle $\alpha$. We use here an additive notation, i.e.~for an angle coordinate $t \in \R/\Z \simeq \mathbb{S}^1$ we write $R_\alpha(t) = t + \alpha \text{ mod } 1$. We assume $\alpha \in \R \setminus \Q$. By $d(\cdot, \cdot)$ we denote the standard rotation-invariant metric on $\mathbb{S}^1$.

For the system $(\mathbb{S}^1, \Leb_{\mathbb S^1}, R_\alpha)$ we consider Lipschitz observables $h\colon\mathbb S^1 \to \R$ and the corresponding $1$-delay coordinate maps $\phi \colon\mathbb S^1 \to \R$. Note that $1$-delay coordinate maps are equal to the observables, i.e.~$\phi = h$.

\begin{prop}\label{prop:as_predict_implies_constant}
Suppose that for a Lipschitz function $h \colon \mathbb{S}^1\to \R$, $\nu$-almost every $y \in \R$ is predictable for the $1$-delay coordinate map $\phi =h$, where $\nu = \phi_*\Leb_{\mathbb S^1}$. Then $h$ is constant.
\end{prop}
\begin{proof}Take $h$ as in the proposition. The proof that $h$ is constant is divided into four parts, described by the following claims.

\begin{Claim1}
There exists a set $B \subset \mathbb{S}^1$ of full $\Leb_{\mathbb S^1}$-measure, with the following property: if $t_1, t_2 \in B$ and $h(t_1) = h(t_2)$, then $h(R^n_\alpha t_1) = h(R^n_\alpha t_2)$ for every $n \geq 0$.
\end{Claim1}
 
\noindent
For the proof of the above claim, consider the system $\{ \mu_y : y \in \R\}$ of conditional measures of $\Leb_{\mathbb S^1}$ with respect to $\phi = h$, given by Theorem~\ref{thm:top_rohlin}. Let
\[
A = \Big\{ t\in \mathbb{S}^1: h (R_\alpha t) = \int \limits h \circ R_\alpha d\mu_{h(t)} \Big\}.
\]
It follows from Theorem~\ref{thm:top_rohlin} that the map $y \mapsto \int \limits h \circ R_\alpha d\mu_{y}$ is $\nu$-measurable, hence $t \mapsto \int \limits h \circ R_\alpha d\mu_{h(t)}$ is $\Leb_{\mathbb S^1}$-measurable. Consequently, $A$ is a $\Leb_{\mathbb S^1}$-measurable set. By Theorem~\ref{thm:top_rohlin},  
\begin{equation}\label{eq:Leb(A)}
\Leb_{\mathbb S^1}(A) =  \int \limits_\R \mu_y(A) d\nu(y)
\end{equation}
and
\[ \mu_y(A) = \mu_y(A \cap \{h = y\}) = \mu_y \Big( \Big\{ t\in\mathbb{S}^1: h(t) = y \text{ and } h(R_\alpha t) = \int \limits h \circ R_\alpha d\mu_{y} \Big\} \Big). \] 
Since $\nu$-almost every $y \in \R$ is predictable, Lemma~\ref{lem:conv_predict} implies that the function $h \circ R_\alpha$ is constant $\mu_y$-almost surely for $\nu$-almost every $y \in \R$, hence $\mu_y(A) = 1$ for $\nu$-almost every $y \in \R$. Therefore, \eqref{eq:Leb(A)} gives $\Leb_{\mathbb S^1}(A) = 1$.

Let 
\[
B = \bigcap \limits_{n=0}^{\infty} R_\alpha^{-n}(A).
\]
Then $B$ has full $\Leb_{\mathbb S^1}$-measure. Moreover, the definition of $A$ implies
that if $t_1, t_2 \in A$ and $h(t_1) = h(t_2)$, then $h(R_\alpha t_1) = h(R_\alpha t_2)$. Therefore, if $t_1, t_2 \in B$ and $h(t_1) = h(t_2)$, then $h(R^n_\alpha t_1) = h(R^n_\alpha t_2)$ for every $n \geq 0$.

\begin{Claim2}
If $t_1, t_2 \in B$ and $h(t_1) = h(t_2)$, then $h(t_1 + s) = h(t_2 + s)$ for every $s \in \mathbb{S}^1$.
\end{Claim2}

\noindent
In order to prove the claim, assume that $t_1, t_2 \in B$ and $h(t_1) = h(t_2)$. Fix $s \in \mathbb{S}^1$. Since $\alpha \notin \Q$, every orbit under $R_\alpha$ is dense in $\mathbb{S}^1$, so there exists a sequence $n_k \to \infty$ with $R^{n_k}_\alpha t_1 \to t_1+s$ as $k \to \infty$. Then $R^{n_k}_\alpha t_2 \to t_2+s$. As $t_1, t_2 \in B$ and $h(t_1)=h(t_2)$, by Claim~1 we have $h(R^{n_k}_\alpha t_1) = h(R^{n_k}_\alpha t_2)$, hence the continuity of $h$ gives $h(t_1 + s) = h(t_2 + s)$. 

\begin{Claim3}
For every $\eps>0$, there exist $t_1, t_2 \in B$ such that $0 < d(t_1, t_2) < \eps$ and $h(t_1) = h(t_2)$.
\end{Claim3}
\noindent

To prove Claim 3, note first that it holds trivially if the set $h^{-1}\left( \left\{\inf h \right\} \right)$ has non-empty interior. Otherwise, fix a small $\eps > 0$ and take $t_0 \in \mathbb{S}^1$ such that $h(t_0) = \inf h$. Then by the continuity of $h$, there exist disjoint open arcs $I, J \subset \mathbb{S}^1$ of length smaller than $\eps/2$, such that $\overline I \cap \overline J = \{ t_0 \}$ and their images $h(I),\ h(J)$ are intervals of positive length with $\overline{h(I)} = \overline{h(J)} = K$ for some closed, non-degenerate interval $K \subset \R$. As $B$ is of full $\Leb_{\mathbb S^1}$-measure and $h$ is Lipschitz, $h(I \cap B)$ and $h(J \cap B)$ both have full Lebesgue measure in $K$, hence $h(I \cap B) \cap h(J \cap B) \neq \emptyset$. This proves the claim.

\begin{Claim4}
$h$ is constant.
\end{Claim4}

\noindent
For the proof of Claim 4, fix a small $\delta>0$. As $h$ is uniformly continuous, there exists $\eps>0$ such that $|h(t) - h(t')| < \delta$ whenever $d(t,t') < \eps$. According to Claim~3, there exist $t_1, t_2 \in B$ such that $0<d(t_1, t_2) < \eps$ and $h(t_1) = h(t_2)$. Let $\beta = t_2 - t_1\text{ mod }1$ and note that $\beta \ne 0$, $|\beta| < \eps$. Applying inductively Claim~2 to $t_1, t_2$ with $s = \beta, \ldots, (n-1)\beta \text{ mod }1$, for $n \in \N$, we obtain $h(t_1) = h(t_1 + \beta\text{ mod }1) = \cdots = h(t_1 + n\beta\text{ mod }1)$. Again by Claim~2, we arrive at $h(0) = h(n\beta\text{ mod }1)$ for $n \in \N$. 

Take $t \in \mathbb{S}^1$. As $|\beta| < \eps$, for every $t \in \mathbb{S}^1$ there exists $n \in \N$ such that $d(t, n\beta\text{ mod }1) < \eps$. For such $n$ we have $|h(t) - h(0)| = |h(t) - h(n\beta\text{ mod }1)| < \delta$. As $\delta$ was arbitrary, we have $h(t) = h(0)$. Therefore, $h$ is constant.
\end{proof}

\begin{rem}
In \cite[Example 3.5]{BGS19} it is shown that there does not exist a Lipschitz map $h \colon \mathbb{S}^1\to \R$ which is injective on a set of full $\Leb_{\mathbb{S}^1}$-measure. However, it may still happen that for certain Lipschitz transformations $T \colon \mathbb{S}^1\to \mathbb{S}^1$ almost sure predictability holds for every $h$, e.g.~if $T$ is the identity.
\end{rem}

\begin{cor}\label{cor:unnatural_counterexample}
Let $X \subset \R^N$ be a compact set with a Borel probability measure $\mu$ and let $T\colon X \to X$ be an injective Lipschitz map, such that 
\[
(\supp \mu, \mu, T|_{\supp \mu}) = (X_0, \mu_0, T_0), 
\]
where $X_0 = \{p_0\} \cup \mathbb{S}^1$ for a circle $\mathbb{S}^1 \subset \R^N$ and $p_0 \in \R^N \setminus \mathbb{S}^1$, 
\[
\mu_0 = \frac 1 2 \delta_{p_0} + \frac 1 2 \Leb_{\mathbb{S}^1},
\]
and $T_0 \colon X_0 \to X_0$, such that $T_0(p_0)=p_0$ and $T_0$ is an irrational rotation $R_\alpha$ on $\mathbb{S}^1$. Set $\nu = \phi_*\mu$. Then $\idim(\mu) = 1/2$ and the only Lip\-schitz observables $h \colon X \to \R$, such that $\nu$-almost every $y \in \R^k$ is predictable for the $1$-delay coordinate map $\phi =h$, are the ones constant on $\mathbb{S}^1$.
Consequently, for a prevalent set of Lipschitz observables, there is a positive $\nu$-measure set of non-predictable points. In particular, the set of  Lipschitz observables $h\colon X \to \R$ for which $\nu$-almost every point of $\R$ is predictable, is not prevalent.
\end{cor}
\begin{proof}
The fact $\idim(\mu) = \idim(\mu_0) = 1/2$ follows from the definition of the information by a direct checking. The assertion that only observables constant on $\mathbb{S}^1$ give almost sure predictability is an immediate consequence of Proposition~\ref{prop:as_predict_implies_constant}. Consider now the space $\Lip(X)$ of all Lipschitz observables $T\colon X \to X$, with the Lipschitz norm $\|h\|_{\Lip}$ (see Subsection~\ref{subsec:prev}), and let $Z \subset \Lip(X)$ be the set of Lipschitz observables which are constant on $\mathbb{S}^1$. Note first that any prevalent set is dense (see \cite[Section 5.1]{Rob11}), while $Z$ is not dense in $\Lip(X)$ in the supremum norm (hence also in the Lipschitz norm). Therefore, $Z$ is not prevalent in $\Lip(X)$. In fact, we can prove more, showing that $\Lip(X) \setminus Z$ is prevalent (note that a subset of the complement of a prevalent set cannot be prevalent, as the intersection of two prevalent sets is prevalent, see \cite{Prevalence92}). 

In order to prove prevalence of $\Lip(X) \setminus Z$, we can assume that the circle $\mathbb{S}^1 \subset X$ is of the form $\mathbb{S}^1 = \{ (x_1, \ldots, x_N) \in \R^N : x_1^2 + x_2^2 = 1,\ x_3 = 0, \ldots, x_N = 0  \}$. Indeed, an affine change of coordinates in $\R^N$ transforming the circle in $X$ to the circle $\{ (x_1, \ldots, x_N) \in \R^N : x_1^2 + x_2^2 = 1,\ x_3 = 0, \ldots, x_N = 0 \}$ induces a linear isomorphism between the corresponding spaces of Lipschitz observables. Like in Theorem~\ref{thm:convergence_takens}, we show the prevalence of $\Lip(X) \setminus Z$ with the probe set equal to a linear basis of the space of real polynomials of $N$ variables of degree at most $1$. In other words, we should check that for any $h \in \Lip(X)$, we have $h + \alpha_0 + \alpha_1 h_1 + \cdots + \alpha_N h_N \notin Z$ for Lebesgue-almost every $\alpha = (\alpha_0, \ldots, \alpha_N) \in \R^{N+1}$, where $h_j(x_1, \ldots, x_N) = x_j$, $j = 1, \ldots, N$. Let $e_1, \ldots, e_N$ be the standard basis of $\R^N$. If $h + \alpha_0 + \alpha_1 h_1 + \cdots + \alpha_N h_N \in Z$, then evaluating at $e_1,e_2 \in \mathbb{S}^1$ gives
\[ h(e_1) + \alpha_0 + \alpha_1 = h(e_2) + \alpha_0 + \alpha_2. \]
Therefore $\alpha_1 = \alpha_2 + h(e_2) - h(e_1)$, so $\alpha$ belongs to an affine subspace of $\R^{N+1}$ of codimension one. It follows that given $h \in \Lip(X)$, we have $h + \alpha_0 + \alpha_1 h_1 + \cdots + \alpha_N h_N \in Z$ for $(\alpha_0, \ldots, \alpha_N)$ in a set of zero Lebesgue measure in $\R^{N+1}$, which ends the proof.
\end{proof}

\subsection{\boldmath Construction of the diffeomorphism $f\colon \mathbb S^2 \to \mathbb S^2$}
\label{subsec:S^2}

In this subsection we construct a smooth diffeomorphism $f$ of $\mathbb S^2 \simeq \R^2 \cup \{\infty\}$ with the invariant unit circle $S$ containing two fixed points $p, q$, such that the trajectories of all points in $\R^2 \setminus \{0,0)\}$ spiral towards the invariant unit circle $S$, spending most of the time in small neighbourhoods of $p$ and $q$.

We consider points $(x,y) \in \R^2$ in polar coordinates, i.e.~$x = r \cos \varphi$, $y = r \sin \varphi$ for $r \in [0, +\infty)$, $\varphi \in \R$. Let
\[
f(r\cos \varphi, r\sin \varphi) = (R(r) \cos \Phi(r, \varphi), R(r) \sin \Phi(r, \varphi))
\]
for
\[
R(r) = r + \varepsilon \frac{r(1-r)^3}{1+r^4}, \qquad \Phi(r, \varphi) = \varphi + \varepsilon \theta(\varphi) + (1-r)^2 \eta(r),
\]
where $\varepsilon > 0$ is a small constant, $\theta\colon \R \to [0, +\infty)$ is a $\pi$-periodic $C^\infty$-function such that $\theta(\varphi) = \varphi^2$ for $\varphi \in (-\pi/4, \pi/4)$ and $\theta$ has no zeroes except for $k\pi$, $k \in \Z$, while $\eta \colon [0, +\infty) \to [0, +\infty)$ is a $C^\infty$-function such that $\eta|_{[\frac{1}{2}, \frac{3}{2}]} \equiv 1$, $\eta > 0$ on $(0,\infty)$ and $\lim_{r \to 0^+} (1-r)^2\eta(r) = \lim_{r \to +\infty} (1-r)^2\eta(r) = 0$ (the role of $\eta$ is to ensure that $f$ extends to a $C^{\infty}$-diffeomorphism of the sphere). The following two lemmas are elementary.

\begin{lem}\label{lem:R} For sufficiently small $\varepsilon >0$, the function $R$ has the following properties.
	\begin{itemize}
		
		\item[(a)] $R$ is an increasing homeomorphism of $[0,+\infty)$.
		\item[(b)] $R(0) = 0$, $R(r) > r$ for $r \in (0,1)$, $R(1) = 1$ and $R(r) < r$ for $r \in (1, +\infty)$.
		\item[(c)] Near $r = 1$, $R$ has the Taylor expansion $R(r) = 1 + r-1 -\frac{\varepsilon}{2} (r-1)^3 + \cdots$.
	\end{itemize}
\end{lem}

\begin{lem}\label{lem:Phi} For sufficiently small $\varepsilon >0$, the function $\Phi$ has the following properties.
	\begin{itemize}
		\item[(a)] $\Phi(r, \varphi) > \varphi$ for $r \in ((0,1) \cup (1, +\infty))$.
		\item[(b)] For given $r \in (0, +\infty)$, the function $\varphi \mapsto \Phi(r, \varphi)$ is strictly increasing.
		\item[(c)] For the function $\varphi \mapsto \Phi(1, \varphi) \mod 2\pi$, the points $0, \pi$ are the unique fixed points and the intervals $(0, \pi), (\pi, 2\pi)$ are invariant.
	\end{itemize}
\end{lem}

Let
\[
\B = \{(x,y) \in \R^2 : \|(x,y)\| < 1\}, \qquad S= \{(x,y) \in \R^2 : \|(x,y)\| = 1\},
\]
where $\| \cdot \|$ denotes the Euclidean norm. For sufficiently small $\varepsilon$, the function $f$ defines a $C^\infty$-diffeomorphism of $\R^2$, such that the unit disc $\B$, the unit circle $S$ and the complement of $\overline \B$ are $f$-invariant. Compactifying $\R^2$ to the Riemann sphere $\mathbb S^2 \simeq \R^2 \cup \{\infty\}$ and putting $f(\infty) = \infty$, we extend $f$ to a $C^\infty$-diffeomorphism of $\mathbb S^2$ with fixed points at $(0, 0)$ and $\infty$. 
Another two fixed points,
\[
p = (1,0), \qquad q = (-1, 0),
\]
corresponding to the fixed points described in Lemma~\ref{lem:Phi}(c), are located in the unit circle $S$.

Now we analyse the behaviour of the orbits of points $(x, y) \in \mathbb S^2$ under $f$. By Lemma~\ref{lem:Phi}, if $(x, y) = (\cos \varphi_0, \sin \varphi_0) \in S$ for some $\varphi_0 \in \R$, then $f^n(x, y)$ tends to $p$ (resp.~to $q$) as $n \to \infty$ for $\varphi_0 \in (-\pi, 0]$ $\text{mod }2\pi$ (resp.~$\varphi_0 \in (0, \pi]$ $\text{mod }2\pi$). Suppose now  $(x, y) \in \mathbb S^2 \setminus S$. Recall that the points $(0, 0)$ and $\infty$ are fixed, so we can assume $(x, y) \in \mathbb \R^2 \setminus (S \cup \{(0,0)\})$. Then
\[
(x, y) = (r_0 \cos\varphi_0,r_0 \sin\varphi_0)
\]
for $r_0 \in \R \setminus \{1\}$, $\varphi_0 \in \R$. The goal of this subsection is to prove
\[ \lim \limits_{N \to \infty} \frac{1}{N} \sum \limits_{n=0}^{N-1} \delta_{f^n(x, y)} = \frac 1 2 \delta_p + \frac 1 2 \delta_q \]
in the sense of weak-$^*$ convergence (see Corollary \ref{cor:N_i}). To this aim, we find the asymptotics of the subsequent times spent by the iterates of $(x, y)$ in small neighbourhoods of the points $p$ and $q$. We will make calculations only for the case 
\[
r_0 \in (0,1),
\]
since the functions $R$, $\Phi$ are defined such that the behaviour of the trajectories in the case of points $r_0 > 1$ is symmetric (see Remark~\ref{rem:symmetry}). From now on, we fix the initial point $(x, y) = (r_0 \cos\varphi_0,r_0 \sin\varphi_0)$ with $r_0 \in (0,1)$ and allow all the constants appearing below to depend on this point. For $n \in \N$ let
\[
r_n = R^n(r_0)
\]
and define inductively
\[
\varphi_{n+1} = \Phi(r_n, \varphi_n). 
\]
Then 
\[
f^n(r_0\cos\varphi_0, r_0\sin\varphi_0) = (r_n \cos\varphi_n, r_n \sin\varphi_n).
\]
For convenience, set
\[
\rho_n = 1 -r_n
\]
and note that by Lemma~\ref{lem:R}, $\rho_n$ decreases to $0$ as $n \to \infty$.

\begin{lem}\label{lem:R^n} We have
	\[
	\rho_n = \frac{a + o(1)}{\sqrt{n}}
	\]
	as $n \to \infty$ for some $a > 0$. Moreover, for every $0 \le k \le n$,
	\[ \frac{k}{cn^{3/2}} \le \rho_n - \rho_{n+k} \le \frac{ck}{n^{3/2}},
	\]
where $c > 0$  is independent of $n$ and $k$.
\end{lem}
\begin{proof} By Lemma~\ref{lem:R}, we have $\rho_n \searrow 0^+$ as $n \to \infty$ and
	\[
	\rho_{n+1} = \rho_n - \frac{\varepsilon}{2} \rho_n^3 + \cdots
	\]
	for $\rho_n$ close to $0$. Hence, the first assertion follows from the standard analysis of the behaviour of an analytic map near a parabolic fixed point, see e.g.~\cite[Lemma~10.1]{Milnor-book}. To check the second one, note that there exists a univalent holomorphic map $\psi\colon V \to \mathbb C$ (Fatou coordinate) on a domain $V \subset \mathbb C$ containing $\rho_n$ for large $n$, such that $\psi(V)$ contains a half-plane $\{z \in \mathbb C: \Re(z) > c_0\}$ for some $c_0 \in \R$ and 
	\[
	\psi(\rho_{n+1}) = \psi(\rho_n) + 1
	\]
	(see e.g.~\cite[Theorem~10.9]{Milnor-book}). Let 
	\[                                                                                                                                                                                                          z_n = \psi(\rho_n)                                                                                                                                                                                                                                                                                                                                                                                                                                                                                                                                                                      \]
	for large $n$ and take $n_0$ with $\Re(z_{n_0}) > c_0$. Then  $\psi^{-1}$ is defined on 
	\[ 
	D = \{z \in \mathbb C: |z- z_{n+k}| < n+k-n_0\}
	\]
	for large $n$, and $z_{\lfloor n/2\rfloor}, z_n \in D'$ for
	\[
	D' = \{z \in \mathbb C: |z- z_{n+k}| \leq n+k-\lfloor n/2\rfloor\}.
	\]
	Since $k \le n$, the ratio of the radius of $D'$ to the radius of $D$ is at most $\frac{(3/2)n+1}{2n-n_0}$, which tends to $3/4$ as $n \to \infty$. Moreover,
	\[
	\frac{|z_{n+k} - z_n|}{|z_n - z_{\lfloor n/2\rfloor}|} = \frac{k}{n - \lfloor n/2\rfloor}
	\]
	Therefore, by the Koebe distortion theorem (see e.g.~\cite[Theorem~1.6]{Carleson-book}),
	\[
	\frac{1}{c} \frac{k}{n}< \frac{\rho_n - \rho_{n+k}}{\rho_{\lfloor n/2\rfloor} - \rho_n} < c \frac{k}{n}
	\]
	for some constant $c > 0$. Since $\sqrt{n}(\rho_{\lfloor n/2\rfloor} - \rho_n) \to \sqrt{2} - 1$ as $n \to \infty$ by the first assertion of the lemma, this ends the proof.
\end{proof}

{\begin{convention}
Within subsequent calculations, we will $a_n \asymp b_n$ for sequences $a_n, b_n$,  if $\frac 1 c < \frac{a_n}{b_n} < c$, where $c > 0$ is independent of $n$.
\end{convention} 

\begin{lem}\label{lem:parab}
	Suppose
	\[
	x_{n+1} = x_n + a x_n^2
	\]
	for $n\in \Z$ and some $a > 0$. Then for given $x_0 < 0$ $($resp.~$x_0 > 0)$ sufficiently close to $0$, we have
	\[
	x_n \asymp -\frac{1}{n} \qquad \Big(\text{resp. } x_{-n} \asymp \frac{1}{n}\Big)
	\]
	for $n \in \N$.
\end{lem}
\begin{proof}
	Follows directly from \cite[Lemma~10.1]{Milnor-book}.
\end{proof}

By Lemmas~\ref{lem:R}--\ref{lem:R^n}, the trajectory of $(x, y)$ approaches the unit circle $S$, spiralling an infinite number of times near $S$ and slowing down near the fixed points $p$ and $q$. In fact, the definitions of the functions $R$, $\Phi$ easily imply that $p$ and $q$ are in the limit set of the trajectory. In particular, for a fixed $\delta > 0$ (which is small enough to satisfy several conditions, specified later), the trajectory visits infinitely number of times the $\delta$-neighbourhoods of $p$ and $q$, defined respectively by
\begin{equation}\label{eq:UAB_def}
\begin{aligned}
U_p &= \{(r\cos\varphi, r\sin\varphi): r \in (1-\delta, 1+\delta),\, \varphi \in (-\delta, \delta)\},\\
U_q &= \{(r\cos\varphi, r\sin\varphi): r \in (1-\delta, 1+\delta),\, \varphi \in (\pi -\delta, \pi + \delta)\}.
\end{aligned}
\end{equation}
Hence, for $i \in \N$ we can define $N_{p, i}$ (resp.~$N_{q, i}$) to be the time spent by the trajectory during its $i$-th visit in $U_p$ (resp.~$U_q$). More precisely, set $n^+_{p,0} = 0$ and define inductively 
\begin{align*}
n^-_{p, i} &= \min\{n \ge n^+_{p, i-1}: (r_n\cos\varphi_n, r_n\sin\varphi_n) \in U_p\},\\
n^+_{p, i} &= \min\{n \ge n^-_{p, i}: (r_n\cos\varphi_n, r_n\sin\varphi_n) \notin U_p\},\\
N_{p, i} &= n^+_{p, i} - n^-_{p, i}
\end{align*}
for $i \geq 1$. Define $n^-_{q, i}$, $n^+_{q, i}$, $N_{q, i}$ analogously. By Lemmas~\ref{lem:R} and~\ref{lem:Phi}, if $\delta>0$ is chosen small enough, then
\begin{equation}\label{eq:n<}
0 < n^-_{p,1} < n^+_{p,1} < n^-_{q,1} < n^+_{q,1} <  \cdots < n^-_{p, i} < n^+_{p, i} < n^-_{q, i} < n^+_{q, i} < \cdots
\end{equation}
or
\[
0 < n^-_{q,1} < n^+_{q,1} < n^-_{p,1} < n^+_{p,1}  < \cdots < n^-_{q, i} < n^+_{q, i} < n^-_{p, i} < n^+_{p, i}  < \cdots,
\]
depending on the position of the point $(x, y)$. To simplify notation, we assume that \eqref{eq:n<} holds. Again by Lemmas~\ref{lem:R} and \ref{lem:Phi}, we obtain the following.

\begin{lem}\label{lem:n-n} We have
\[
n^-_{q, i} - n^+_{p, i}, \; n^-_{p,i+1} - n^+_{q, i} < N_0
\]
for some constant $N_0 > 0$. In other words, the times spent by the trajectory of $(x, y)$ between consecutive visits in $U_p \cup U_q$ remain uniformly bounded.
\end{lem}

Now we estimate the times spent by the trajectory during its stay in $U_p$ and $U_q$.

\begin{lem}\label{lem:N_i}
	\[
	N_{p, i} \asymp N_{q, i} \asymp i.
	\]
\end{lem}
\begin{proof}
	We prove the lemma by induction. Obviously, we can assume that $i$ is large. Suppose, by induction, 
		\begin{equation}\label{eq:ind-N}
		\frac{j}{C} \le N_{p, j} \le C j, \quad  \frac{j}{C}  \le N_{q, j} \le C j \qquad \text{for } j = 1, \ldots, i-1
	\end{equation}
	for a large constant $C > 1$ (to be specified later). First, we estimate $N_{p, i}$. By Lemma~\ref{lem:n-n},
	\begin{equation}\label{eq:S_n}
	\frac{i^2}{c_1C}  \leq n^-_{p, i} \leq c_1 C i^2 
	\end{equation}
	for some $c_1 > 0$ (we denote by $c_1, c_2, \ldots$ constants independent of $C$.) Obviously, we can assume $\varphi_{n^-_{p, i}} \in [-\pi, \pi)$. Then, since $\delta$ is small and $i$ is large, we have 
	\[
	-\frac{\pi}{4} < -\delta < \varphi_{n^-_{p, i}} < 0.
	\]
	Note that $\rho_{n^-_{p, i}} < \delta$ and the sequence $\rho_n$ is decreasing, so
	\[
	N_{p, i} = \min\{n \ge n^-_{p, i}: \varphi_n \ge  \delta\} - n^-_{p, i}.
	\]
	Recall that if $\varphi_n \in (-\pi/4, \pi/4)$ (in particular, if $n \in [n^-_{p, i}, n^+_{p, i})$), then
	\begin{equation}\label{eq:ind}
	\varphi_{n+1} = \varphi_n + \varepsilon \varphi_n^2 + \rho_n^2.
	\end{equation}
	
	Let
	\[
	\rho^-_i = \frac{1}{C^{2/3}i}, \qquad 
	\rho^+_i = \frac{C^{2/3}}{i}.
	\]
	To estimate the behaviour of the sequence $\varphi_n$ for $n \ge n^-_{p, i}$, we will compare it with the sequences $\varphi^+_n$, $\varphi^-_n$ for $n \ge n^-_{p, i}$, given by 
	\begin{equation}\label{eq:tildeind}
	\varphi^\pm_{n^-_{p, i}} = \varphi_{n^-_{p, i}}, \qquad \varphi^\pm_{n+1} = \varphi^\pm_n + \varepsilon (\varphi^\pm_n)^2 + (\rho^\pm_i)^2.
	\end{equation}
	First, we will analyse the behaviour of the sequences $\varphi^\pm_n$ and then show that they provide upper and lower bounds for $\varphi_n$. 
	By definition, $\varphi_{n^-_{p, i}}^\pm \in (-\delta, 0)$ and $\varphi_n^\pm$ increases to infinity as $n \to \infty$. Hence, we can define 
	\[
	N^\pm_i = \min\{n \ge n^-_{p, i+1}: \varphi^\pm_n \ge \delta\} - n^-_{p, i}.
	\]
	to be the time which the sequence $\varphi^\pm_n$ spends in $(-\delta, \delta)$. Since $\rho^-_i < \rho^+_i$, we have $\varphi^-_n \le \varphi^+_n$ and $N^+ \le N^-$. 
	Set
	\begin{align*}
	k^\pm_1 &= \min\left\{n \in [n^-_{p, i}, n^-_{p, i} + N^\pm_i] : \varphi^\pm_n > - \frac{\rho^\pm_i}{\sqrt{\varepsilon}}\right\},\\
	k^\pm_2 &= \min\left\{n \in [k^\pm_1, n^-_{p, i} + N^\pm_i] : \varphi^\pm_n >  \frac{\rho^\pm_i}{\sqrt{\varepsilon}}\right\}.
	\end{align*}
	Note that for $n \in [n^-_{p, i}, k^\pm_1) \cup [k^\pm_2, N^\pm_i + n^-_{p, i})$ we have $\varepsilon(\varphi^\pm_n)^2 \ge (\rho^\pm_i)^2$, so
	\[
	\varphi^\pm_n + \varepsilon(\varphi^\pm_n)^2 \le \varphi^\pm_{n+1} \le  \varphi^\pm_n + 2\varepsilon (\varphi^\pm_n)^2.
	\]
	Hence, by Lemma~\ref{lem:parab}, 
	\[
	k^\pm_1 - n^-_{p, i} \asymp N^\pm_i + n^-_{p, i} - k^\pm_2 \asymp \frac{1}{\rho^\pm_i}.
	\]
	On the other hand, for $n \in [k^\pm_1, k^\pm_2)$ we have $\varepsilon(\varphi^\pm_n)^2 \le (\rho^\pm_i)^2$, so
	\[
	\varphi^\pm_n + (\rho^\pm_i)^2 \le \varphi^\pm_{n+1} \le  \varphi^\pm_n + 2 (\rho^\pm_i)^2,
	\]
	which implies
	\[
	k^\pm_2 - k^\pm_1 \asymp \frac{1}{\rho^\pm_i}.
	\]
	Hence,
	\[
	\frac{i}{c_2 C^{2/3}} = \frac{1}{c_2\rho^+_i} \le N^+_i \le N^-_i \leq  \frac{c_2}{\rho^-_i} = c_2 C^{2/3} i
	\]
	for some $c_2 > 0$. If $C$ is chosen sufficiently large, then this yields
	\begin{equation}\label{eq:N_pm}
	\frac{i}{C} \le N^+_i \le N^-_i \leq C i.
	\end{equation}
	
	Now we show by induction that
	\begin{equation}\label{eq:phipm}
	\varphi^-_n \le \varphi_n \le \varphi^+_n 
	\end{equation}
	for $n \in [n^-_{p, i}, n^-_{p, i} + \min(N_{p, i}, N^-_i)]$. 
	To do it, note that for $n = n^-_{p, i}$ we have equalities in \eqref{eq:phipm}. 
	Suppose, by induction, that \eqref{eq:phipm} is satisfied for some $n \in [n^-_{p, i}, n^-_{p, i} + \min(N_{p, i}, N^-_i))$. Then
	by \eqref{eq:ind} and \eqref{eq:tildeind},
	\[
	\varphi_{n+1} - \varphi^\pm_{n+1} = (\varphi_n - \varphi^\pm_n)(1 + \varepsilon (\varphi_n + \varphi^\pm_n)) +  \rho_n^2 - (\rho^\pm_i)^2,
	\]
	where $1 + \varepsilon (\varphi_n + \varphi^\pm_n) >1 -2\varepsilon\delta > 0$. Moreover, by Lemma~\ref{lem:R^n}, \eqref{eq:S_n} and \eqref{eq:N_pm}, there exists a constant $c_3 > 0$, such that
	\[
	\frac{1}{c_3 \sqrt{C} \, i} \le  \rho_n \le \frac{c_3 \sqrt{C}}{i},
	\]
	which gives
	\[
	\rho^-_i \le  \rho_n \le \rho^+_i,
	\]
	provided $C$ is chosen sufficiently large. Therefore, the sign of $\varphi_{n+1} - \varphi^\pm_{n+1}$ is the same as the one of $\varphi_n - \varphi^\pm_n$, which provides the induction step and proves \eqref{eq:phipm}.
	
	Using \eqref{eq:phipm}, we can show 
	\begin{equation}\label{eq:N<N<N}
	N^+ _i\le N_{p, i} \le N^-_i.
	\end{equation}
 Indeed, if  $N_{p, i} > N_i^-$, then by \eqref{eq:phipm},
	\[
	\delta \le \varphi^-_{n^-_{p, i} + N^-_i} \le \varphi_{n^-_{p, i} + N^-_i},
	\]
	so $n^+_{p, i} \le n^-_{p, i} + N^-_i$, which is a contradiction. Hence, $N_{p, i} \le N_i^-$, and then \eqref{eq:phipm} gives
	\[
	\delta \le \varphi_{n^+_{p, i}} \le \varphi^+_{n^+_{p, i}},
	\]
	which implies \eqref{eq:N<N<N}. By \eqref{eq:N_pm} and \eqref{eq:N<N<N},
	\[
	\frac{i}{C} \le N_{p, i} \leq C i,
	\]
	which completes the inductive step started in \eqref{eq:ind-N} and shows $N_{p, i} \asymp i$. 
	
	To show $N_{q, i} \asymp i$, note that if $\varphi_n \in (3\pi/4 , 5\pi/4)$, then for $\tilde \varphi_n = \varphi_n - \pi$ we have
	\[
	\tilde \varphi_{n+1} = \tilde \varphi_n + \varepsilon \tilde \varphi_n^2 + \rho_n^2.
	\]
	Moreover, by the proved assertion $N_{p, i} \asymp i$ and Lemmas~\ref{lem:R^n} and~\ref{lem:n-n}, we have $n^-_{q, i} \asymp n^-_{p, i}$ and $\rho_{n^-_{q, i}}\asymp\rho_{n^-_{p, i}}$. Using this, one can show $N_{q, i}\asymp i$ by repeating the proof in the case of $N_{p, i}$.
\end{proof}

A more accurate comparison of $N_{p, i}$ and $N_{q, i}$ is presented below.

\begin{lem}\label{lem:N_i-N_i} There exists $M > 0$ such that
	\[
	|N_{p, i}-N_{q, i}| < M
	\]
	for all $i \geq 1$.
\end{lem}
\begin{proof} Take a large $i \in \N$. Let
	\[
	(\eta_n, \psi_n) = f^n(r_{n_{p, i}^-}, \varphi_{n_{p, i}^-}), \qquad (\tilde\eta_n, \tilde \psi_n) = f^n(r_{n_{q, i}^-}, \varphi_{n_{q, i}^-} - \pi) 
	\]
	and
	\[
	\sigma_n = 1 -\eta_n = \rho_{n+n_{p, i}^-}, \qquad \tilde \sigma_n = 1 -\tilde \eta_n  = \rho_{n+n_{q, i}^-}
	\]
	for $n \ge 0$. Subtracting multiplicities of $2\pi$, we can assume $\psi_0, \tilde \psi_0 \in [-\pi, \pi)$, so in fact
	\[
	-\delta < \psi_0, \tilde \psi_0 < 0.
	\]
	By definition,
	\begin{equation}\label{eq:ind2}
	\psi_{n+1} = \psi_n + \varepsilon\psi_n^2 + \sigma_n^2, \qquad  \tilde \psi_{n+1} =  \tilde \psi_n +  \varepsilon\tilde \psi_n^2 +  \tilde \sigma_n^2
	\end{equation}
	as long as $\psi_n, \tilde\psi_n  < \pi/4$. It follows that
	\[
	N_{p, i} = \min\{n \ge 0: \psi_n \ge \delta\}, \qquad N_{q, i} = \min\{n \ge 0: \tilde\psi_n \ge \delta\}.
	\] 
	Note that \eqref{eq:ind2} holds for $n \le \min(N_{p, i}, N_{q, i})+1$. To prove the lemma, we will carefully compare the behaviour of the sequences $\psi_n$ and $\tilde\psi_n$. First, note that 
	\begin{equation}\label{eq:psi_0}
	\tilde\psi_0 \le \psi_2 \le \tilde\psi_4
	\end{equation}
	provided $i$ is sufficiently large (because then $\sigma_n, \tilde\sigma_n$ are small compared to $\eps$ and $\delta$). Note also that since $\rho_n$ is decreasing, we have 
	\begin{equation}\label{eq:sigma>}
	\sigma_{n+2} > \tilde \sigma_n
	\end{equation}
	for every $n \ge 0$. By \eqref{eq:ind2},
	\[
	\psi_{n+3} - \tilde \psi_{n+1} = (\psi_{n+2} - \tilde \psi_n) (1 + \varepsilon(\psi_{n+2} + \tilde \psi_n)) + \sigma_{n+2}^2 - \tilde \sigma_n^2
	\]
	for $n \le \min(N_{p, i}-2, N_{q, i})$, where $\varepsilon(\psi_{n+2} + \tilde \psi_n) < \varepsilon \pi/2 < 1$. Hence, by induction, using \eqref{eq:psi_0} and \eqref{eq:sigma>}, we obtain
	\begin{equation}\label{eq:psi>tilde}
	\psi_{n+2} \ge \tilde \psi_n
	\end{equation}
	for $n \in [0, \min(N_{p, i}-2, N_{q, i})+1]$. In particular,
	\[
	N_{p, i} < N_{q, i}+2 \qquad \text{or} \qquad \psi_{N_{q, i}+2} > \tilde \psi_{N_{q, i}} \ge \delta,
	\]
	which gives
	\begin{equation}\label{eq:N<}
	N_{p, i} \leq N_{q, i} + 2.
	\end{equation}
	
	The proof of the opposite estimate is more involved, so let us first present its sketch. We fix a number $k$ such that (roughly speaking) $\psi_k \approx 1/i$. Then we show inductively $\tilde \psi_{n+2} \ge \psi_n - cn/i^3$ for $n \le k$ and some constant $c >0$ (see~\eqref{eq:ind_psi}). This gives $\tilde \psi_{k+2} \ge \psi_k - c'/i^2$ for some $c' > 0$ (see \eqref{eq:c_3}). By the definition of $k$, we check that for sufficiently large constant $M > 0$ we have $\tilde \psi_{k+M} \ge \psi_k + c''M/i^2$ for some $c'' > 0$. With this starting condition, we inductively show $\tilde \psi_{n+M} \ge \psi_n + c''M/i^2$ for $n \in [k, N_{p, i}]$ (see~\eqref{eq:ind_psi2}). This provides $\tilde \psi_{N_{p, i}+M} \ge \psi_{N_{p, i}} \ge \delta$, so $N_{q, i} \le N_{p, i} + M$. 
	
	Now let us go into the details of the proof. By Lemmas~\ref{lem:n-n} and~\ref{lem:N_i}, we have 
	\begin{equation}\label{eq:Nasymp}
	n_{p, i}^- \asymp n_{q, i}^- \asymp i^2, \qquad N_{p, i} \asymp N_{q, i} \asymp i,
	\end{equation}
	so by Lemma~\ref{lem:R^n},
	\begin{equation}\label{eq:sigma_n}
	\sigma_n \le \frac{c_1}{i}, \qquad \sigma_n^2 - \tilde \sigma_{n+2}^2 = (\sigma_n + \tilde \sigma_{n+2})(\sigma_n - \tilde \sigma_{n+2})  \le
	\frac{c_1}{i^3}
	\end{equation}
	for $n \in [0, N_{q, i}+4]$ and a constant $c_1 > 0$. Let 
	\[
	k = \max \left\{n \in [2, N_{q, i}]: \psi_{n+4} < \frac{b}{i}\right\}
	\]
	for a small constant $b > 0$ (to be specified later). Note that $k \le \min(N_{p, i} - 5, N_{q, i})$, so \eqref{eq:ind2} holds for $n \in [2, k)$.

	We will show by induction that 
	\begin{equation}\label{eq:ind_psi}
	\psi_n - \tilde \psi_{n+2} \le \frac{2c_1n}{i^3}
	\end{equation} 
	for every $n \in [2,k]$. For $n = 2$, \eqref{eq:ind_psi} holds due to \eqref{eq:psi_0}. Suppose it holds for some $n\in [2, k)$. By \eqref{eq:ind2}, we have
	\[
	\psi_{n+1} - \tilde \psi_{n+3} = (\psi_n - \tilde \psi_{n+2}) (1 + \varepsilon(\psi_n + \tilde \psi_{n+2})) + \sigma_n^2 - \tilde \sigma_{n+2}^2,
	\]
	where by \eqref{eq:psi>tilde} and the definition of $k$, $\psi_n + \tilde \psi_{n+2} \le \psi_n + \psi_{n+4} < 2\psi_{n+4} < 2b/i$, so using \eqref{eq:Nasymp}, \eqref{eq:sigma_n} and the inductive assumption (\ref{eq:ind_psi}), we obtain
	\[
	\psi_{n+1} - \tilde \psi_{n+3} \le \frac{2c_1 n}{i^3} \left(1 +  \frac{2\varepsilon b}{i}\right) + \frac{c_1}{i^3} \le \left(2n +  \frac{4\varepsilon b N_{q, i}}{i} + 1\right)\frac{c_1}{i^3} < \frac{(2n + c_2 b  + 1)c_1}{i^3}
	\]
	for some constant $c_2 > 0$. Choosing the constant $b$ in the definition of $k$ sufficiently small, we can assume $c_2 b< 1$, which gives 
	\[
	\psi_{n+1} - \tilde \psi_{n+3} \le \frac{2c_1(n + 1)}{i^3}.
	\]
	This completes the inductive step and proves \eqref{eq:ind_psi}.
	
	By \eqref{eq:Nasymp} and \eqref{eq:ind_psi},
	\begin{equation}\label{eq:c_3}
	\tilde \psi_{k+2} \ge \psi_k - \frac{c_3}{i^2}
	\end{equation}
	for a constant $c_3 > 0$, while (by the definition of $k$),
	\begin{equation}\label{eq:k+5>}
	\psi_{k+5} \ge \frac{b}{i}
	\end{equation}
	and by \eqref{eq:ind2},
	\begin{equation}\label{eq:k+5<}
	\psi_{k+5} = \psi_k + \varepsilon (\psi_k^2 + \cdots + \psi_{k+4}^2) +  \sigma_k^2 + \cdots + \sigma_{k+4}^2  <\psi_k + \frac{5(\varepsilon b^2 + c_1)}{i^2}.
	\end{equation}
	by the definition of $k$, \eqref{eq:ind2} and \eqref{eq:sigma_n}. Using \eqref{eq:c_3}, \eqref{eq:k+5>} and \eqref{eq:k+5<}, we obtain
	\begin{equation}\label{eq:tilde_psi>}
	\tilde \psi_{k+2} \ge \frac{b}{i} - \frac{5(\varepsilon b^2 + c_1)+c_3}{i^2}\ge  \frac{b}{2i}
	\end{equation}
	for large $i$.
	
	Take a large constant $M > 0$. We will show inductively
	\begin{equation}\label{eq:ind_psi2}
	\tilde \psi_{n+M} - \psi_n \ge \frac{M\varepsilon b^2}{5i^2}
	\end{equation}
	for $n \in [k, N_{p, i}]$. By \eqref{eq:ind2}, \eqref{eq:c_3} and \eqref{eq:tilde_psi>}, we have 
	\begin{align*}
	\tilde \psi_{k+M} &\ge \tilde \psi_{k+2} + \varepsilon(\tilde \psi_{k+2}^2 + \cdots +\tilde \psi_{k+M}^2) \ge \tilde \psi_{k+2} + (M-2) \varepsilon\tilde \psi_{k+2}^2 \\
	&\ge \tilde \psi_{k+2} + \frac{(M-2) \varepsilon b^2}{4i^2} \ge \psi_k - \frac{c_3}{i^2}+  \frac{(M-2) \varepsilon b^2}{4i^2} \ge \psi_k + \frac{M\varepsilon b^2}{5i^2},
	\end{align*}
	if $M$ is chosen sufficiently large, so \eqref{eq:ind_psi2} holds for $n = k$. Suppose \eqref{eq:ind_psi2} holds for some $n \in [k, N_{p, i})$. Now \eqref{eq:N<} implies that \eqref{eq:ind2} is valid for $n$, so 
	\[
	\tilde \psi_{n+1+M} - \psi_{n+1} = (\tilde\psi_{n+M} - \psi_n) (1 +  \varepsilon(\tilde\psi_{n+M} + \psi_n)) + \tilde \sigma_{n+M}^2 -\sigma_n^2,
	\]
	where 
	\[
	\tilde\psi_{n+M} + \psi_n > \tilde\psi_{k+M} + \psi_k> \tilde\psi_{k+2}
	\]
	for large $i$ by \eqref{eq:k+5>} and \eqref{eq:k+5<} (which imply $\psi_k > 0$), while 
	\[
	\tilde \sigma_{n+M}^2 -\sigma_n^2 > - \frac{c_4}{i^3}
	\]
	for a constant $c_4 > 0$ by \eqref{eq:Nasymp} and Lemma~\ref{lem:R^n} (with estimates analogous to the ones in \eqref{eq:sigma_n}). Hence, using  \eqref{eq:tilde_psi>} we obtain
	\[
	\tilde \psi_{M+n+1} - \psi_{n+1} \ge 
	\frac{M\varepsilon b^2}{5i^2}(1 + \varepsilon \tilde \psi_{k+2}) - \frac{c_4}{i^3}
	\ge \frac{M\varepsilon b^2}{5i^2} \left(1 + \frac{\varepsilon b}{2i}\right) - \frac{c_4}{i^3} \ge \frac{M\varepsilon b^2}{5i^2},
	\]
	provided $M$ is chosen sufficiently large. This ends the inductive step and proves \eqref{eq:ind_psi2}. 
	
	By \eqref{eq:ind_psi2},
	\[
	\tilde \psi_{N_{p, i}+M} \ge \psi_{N_{p, i}} \ge \delta,
	\]
	so
	\[
	N_{q, i} \le N_{p, i} + M.
	\]
	This and \eqref{eq:N<} end the proof of the lemma.
\end{proof}
\begin{rem}\label{rem:symmetry} Proving Lemmas~\ref{lem:n-n}--\ref{lem:N_i-N_i}, we have made the calculations for the initial point $(x, y) = (r_0 \cos \varphi_0, r_0 \sin \varphi_0)$ assuming $r_0 \in (0,1)$. In fact, the case $r_0 > 1$ can be treated analogously. This can be seen by noting that $\Phi$ is symmetric with respect to $r$ around the circle $r=1$, while the only properties of $R$ used in the proofs of the lemmas are the ones stated in Lemma~\ref{lem:R}. As the initial terms of the Taylor expansion of $R$ near $r=1$ are symmetric around $1$, we see that an analogue of Lemma~\ref{lem:R^n} holds in the case $r_0 > 1$ and the proof of Lemmas~\ref{lem:n-n}--\ref{lem:N_i-N_i} can be repeated in that case. We conclude that Lemmas~\ref{lem:n-n}--\ref{lem:N_i-N_i} hold for every initial point $(x, y) \in \mathbb S^2 \setminus (S \cup \{(0,0), \infty\})$.
\end{rem}

We summarize the results of this subsection in the following proposition.

\begin{prop}\label{prop:N_i}
	For every $(x, y) \in \mathbb S^2 \setminus (S \cup \{(0,0), \infty\})$ and every $\delta > 0$, if $N_{p, i}(x, y)$ $($resp.~$N_{q, i}(x, y))$ is the time spent by the trajectory of $(x, y)$ under $f$ during its $i$-th visit in the $\delta$-neighbourhood $U_p$ of $p$ $($resp.~$U_q$ of $q)$, defined in \eqref{eq:UAB_def}, then
	\[
	N_{p, i}(x, y) \asymp N_{q, i}(x, y) \asymp i
	\]
	and
	\[ |N_{p, i}(x, y) - N_{q, i}(x, y)| \leq M \]
	for some constant $M>0$, while the times spent by the trajectory between consecutive visits in $U_p \cup U_q$ are uniformly bounded.
\end{prop}

This implies the following. 

\begin{cor}\label{cor:N_i}
	For every $(x, y) \in \mathbb S^2 \setminus (S \cup \{(0,0), \infty\})$,
	\[ \lim \limits_{m \to \infty} \frac{1}{m} \sum \limits_{n=0}^{m-1} \delta_{f^n(x, y)} = \frac 1 2 \delta_p + \frac 1 2 \delta_q \]
	in the sense of weak-$^*$ convergence.
\end{cor}

\begin{proof}
Fix $(x, y) \in \mathbb S^2 \setminus (S \cup \{(0,0), \infty\})$ and $\delta>0$. It is sufficient to prove that for the $\delta$-neighbourhoods $U_p$ and $U_q$, defined in \eqref{eq:UAB_def}, one has
\[ \lim \limits_{m \to \infty} \frac{1}{m} \sum \limits_{n=0}^{m-1} \mathds{1}_{U_p} \big(f^n(x , y) \big) = \lim \limits_{m \to \infty} \frac{1}{m} \sum \limits_{n=0}^{m-1} \mathds{1}_{U_q} \big(f^n(x , y)\big) = \frac{1}{2}. \]
Fix $m \in \N$ and let $i = i(m)$ be the number of visits of $(x , y)$ to $U_p$ completed up to the time $m$, i.e.~let $i$ be the unique number such that
\[ n^-_{p, i} \leq m < n^-_{p,i+1}. \]
Then by Proposition~\ref{prop:N_i}, there exist a constant $c>0$ (independent of $m$) such that
\[ \frac{i^2}{c}\leq\sum \limits_{n=0}^{m-1} \mathds{1}_{U_p}\big(f^n(x , y)\big)  \leq c i^2, \qquad \frac{i^2}{c}\leq \sum \limits_{n=0}^{m-1} \mathds{1}_{U_q}\big(f^n(x , y)\big) \leq c i^2, \]
and
\[ \sum \limits_{n=0}^{m-1} \mathds{1}_{(U_p \cup U_q)^c}\big(f^n(x , y)\big) \leq ci. \]
This implies
\begin{equation}\label{eq:Ni_bound} \frac{2i^2}{c} \leq m \leq 3ci^2
\end{equation}
provided $i$ is large enough (which holds if $m$ is large enough). Therefore,
\[ \lim \limits_{m \to \infty} \frac{1}{m} \sum \limits_{n=0}^{m-1} \mathds{1}_{(U_p \cup U_q)^c}\big(f^n(x , y)\big) = 0  \]
and hence
\begin{equation}\label{eq:AB_sum} \lim \limits_{m \to \infty} \bigg( \frac{1}{m} \sum \limits_{n=0}^{m-1} \mathds{1}_{U_p}\big(f^n(x , y)\big) + \frac{1}{m} \sum \limits_{n=0}^{m-1} \mathds{1}_{U_q}\big(f^n(x , y)\big) \bigg) = 1.
\end{equation}
Proposition \ref{prop:N_i} together with \eqref{eq:Ni_bound} implies
\[ \bigg|\frac{1}{m} \sum \limits_{n=0}^{m-1} \mathds{1}_{U_p} \big(f^n(x , y) \big) - \frac{1}{m}  \sum \limits_{n=0}^{m-1} \mathds{1}_{U_q} \big(f^n(x , y) \big) \bigg| \leq \frac{C}{i}  \]
for a constant $C >0$ (independent of $m$), hence
\begin{equation}\label{eq:AB_diff}
\lim \limits_{m \to \infty }\bigg|\frac{1}{m} \sum \limits_{n=0}^{m-1} \mathds{1}_{U_p} \big(f^n(x , y) \big) - \frac{1}{m}  \sum \limits_{n=0}^{m-1} \mathds{1}_{U_q} \big(f^n(x , y) \big) \bigg| = 0.
\end{equation}
Combining \eqref{eq:AB_sum} with \eqref{eq:AB_diff} finishes the proof (it is enough to notice that if $a_n, b_n$ are sequences of real numbers with $\lim_{n \to \infty} (a_n + b_n) = 1$ and $\lim_{n \to \infty} |a_n - b_n| = 0$, then $\lim_{n \to \infty} a_n = \lim_{n \to \infty} b_n = \frac{1}{2}$). 
\end{proof}

\subsection{\boldmath Construction of the diffeomorphism $T\colon \mathbb S^2 \times \mathbb S^1 \to \mathbb S^2 \times \mathbb S^1$} \label{subsec:S^2xS^1}

Let 
\[
X = \mathbb S^2 \times \mathbb S^1,
\]
where $\mathbb S^2\simeq \R^2 \cup \{\infty\}$ and $\mathbb S^1 \simeq \R / \Z$. We can assume $X \subset \R^N$ for some $N \in \N$. Let 
\[ 
R_\alpha \colon\mathbb{S}^1 \to \mathbb{S}^1, \qquad R_\alpha(t) =  t + \alpha \mod 1, \qquad \alpha \in \R \setminus \Q
\] 
be an irrational rotation. Recall that the normalized Lebesgue measure on $\mathbb{S}^1$ is the unique $R_\alpha$-invariant Borel probability measure. Let 
\[
g\colon \mathbb{S}^1 \to \mathbb{S}^1, \qquad g(t) = t + \frac{1}{100}\sin^2(\pi t) \mod 1.
\]
Note that $g$ is a $C^{\infty}$-diffeomorphism of $\mathbb{S}^1$ with $0$ as the unique fixed point. Moreover, $\lim_{n \to \infty} g^n(t) = 0$ for every $t \in \mathbb{S}^1$. Therefore, $\delta_0$ is the unique $g$-invariant Borel probability measure. Let $f \colon\mathbb{S}^2 \to \mathbb{S}^2$ be the diffeomorphism defined in Subsection~\ref{subsec:S^2}, with the invariant unit circle $S \subset \mathbb S^2$ and the fixed points $p, q \in S$. Fix a small $\delta>0$ and consider the $\delta$-neighbourhoods $U_p, U_q \subset \mathbb S^2$ of $p$ and $q$, respectively, defined in \eqref{eq:UAB_def}. 
Let
\[ 
T \colon X \to X, \qquad T(z,t) = (f(z), h_{z}(t)), \qquad z \in \mathbb{S}^2,\; t \in \mathbb{S}^1, 
\]
where $h_z$ are diffeomorphisms of $\mathbb{S}^1$ depending smoothly on $z \in \mathbb{S}^2$, such that $h_z = g$ for $z \in U_p$, $h_z = R_\alpha$ for $z \in U_q$, and for $z$ outside $U_p \cup U_q$, $h_z$ is defined in any manner which makes $T$ a $C^{\infty}$-diffeomorphism of $X$. \footnote{This is possible since $g$ is smoothly isotopic to identity by the family of maps $g_{\eps}(t) = t + \eps\sin^2(\pi t) \text{ mod } 1$,  $\eps \in [0, \frac{1}{100}]$, while $R_\alpha$ is smoothly isotopic to identity by the family of maps $R_{\eps}(t) = t + \eps$, $\eps \in [0,\alpha]$.}

In view of Corollary~\ref{cor:unnatural_counterexample}, to conclude the proof of Theorem~\ref{thm:counterexample_main}, it is sufficient to show the following.

\begin{thm}\label{thm:natural_counterexample}
The map $T$ has an attractor 
\[
\Lambda = S \times \mathbb{S}^1
\]
with the basin $B(\Lambda) = (\mathbb S^2 \setminus\{(0,0), \infty\}) \times \mathbb{S}^1$ and natural measure 
\[
\mu = \frac 1 2 \delta_{p_0} + \frac 1 2 \Leb_{\mathbb S^1},
\]
where $p_0 = (p,0)$ and $\Leb_{\mathbb S^1}$ is the Lebesgue measure on the circle $\{q\} \times \mathbb S^1$.
\end{thm}

Before proving Theorem \ref{thm:natural_counterexample} we show the following lemma.

\begin{lem}\label{lem:density_zero_limit}
	Let $T \colon X \to X$ be a continuous transformation of a compact metric space. Let $\nu_n,\ n \geq 0$, be a sequence of Borel probability measures on $X$ and let $\mathcal A \subset \N \cup \{0\}$ be a set of asymptotic density zero, i.e.
	\[ \lim \limits_{m \to \infty} \frac{1}{m} \#\{ 0 \leq n < m : n\in \mathcal A \} = 0. \]
	Assume $\nu_{n+1} = T_*\nu_n$ for $n \notin \mathcal A$. Then any weak-$^*$ limit point of the sequence
	\begin{equation*}
	\frac{1}{m} \sum \limits_{n=0}^{m-1} \nu_n
	\end{equation*} is $T$-invariant.
\end{lem}

\begin{proof}
Let $\nu$ be a weak-$^*$ limit of a sequence $\frac{1}{m_k} \sum \limits_{n=0}^{m_k-1} \nu_n$ for some sequence $m_k \nearrow \infty$. Then
\begin{equation}\label{eq:Tnu-nu} T_*\nu - \nu = \lim \limits_{k \to \infty} \frac{1}{m_k} \sum \limits_{n=0}^{m_k-1} (T_*\nu_n - \nu_n)
\end{equation}
and we will prove
\begin{equation}\label{eq:Alim0}
\lim \limits_{k \to \infty}  \Big\| \frac{1}{m_k} \sum \limits_{n=0}^{m_k-1} (T_*\nu_{n} - \nu_n)\mathds{1}_{\mathcal{A}}(n)  \Big\| =0
\end{equation}
and
\begin{equation}\label{eq:Aclim0}
\lim \limits_{k \to \infty}  \Big\| \frac{1}{m_k} \sum \limits_{n=0}^{m_k-1} (T_*\nu_{n} - \nu_n)\mathds{1}_{\mathcal{A}^c}(n) \Big\| = \lim \limits_{k \to \infty}  \Big\| \frac{1}{m_k} \sum \limits_{n=0}^{m_k-1} (\nu_{n+1} - \nu_n)\mathds{1}_{\mathcal{A}^c}(n) \Big\| = 0,
\end{equation} 
where $\| \cdot \|$ stands for the total variation norm. Due to (\ref{eq:Tnu-nu}), this will imply $T_* \nu = \nu$.
For \eqref{eq:Alim0}, we have
\[\lim \limits_{k \to \infty} \Big\| \frac{1}{m_k} \sum \limits_{n=0}^{m_k-1} (T_*\nu_{n} - \nu_n)\mathds{1}_{\mathcal{A}}(n) \Big\| \leq \lim \limits_{k \to \infty} \frac{2}{m_k} \sum \limits_{n=0}^{m_k-1} \mathds{1}_{\mathcal{A}}(n) = 0,\]
as the asymptotic density of $\mathcal{A}$ is zero and all $\nu_n$ and $T_*\nu_n$ are probability measures. For \eqref{eq:Aclim0}, observe that the first equality follows by assumptions, and for a given $n \in \{0, \ldots, m_{k}-2\}$, if both $n$ and $n+1$ are in $\mathcal{A}^c$, then $\nu_{n+1}$ cancels out in the sum $\sum_{n=0}^{m_k-1} (\nu_{n+1} - \nu_n)\mathds{1}_{\mathcal{A}^c}(n)$ and otherwise it appears in the above sum at most once (possibly with a negative sign). The terms $\nu_0$ and $\nu_{m_k}$ appear at most once. Therefore,
\begin{align*}
&\lim \limits_{k \to \infty} \Big\| \frac{1}{m_k} \sum \limits_{n=0}^{m_k-1} (\nu_{n+1} - \nu_n)\mathds{1}_{\mathcal{A}^c}(n) \Big\|\\
&\leq \lim \limits_{k \to \infty} \frac{1}{m_k} \Big( \| \nu_{m_k} \| + \| \nu_0 \| + \sum \limits_{n=0}^{m_k-2} \| \nu_{n+1} \| \big(1- \mathds{1}_{\mathcal{A}^c}(n)\mathds{1}_{\mathcal{A}^c}(n+1)\big) \Big)\\
&= 
\lim \limits_{k \to \infty} \frac{1}{m_k}\Big(2+ \sum \limits_{n=0}^{m_{k}-2} \big(1- \mathds{1}_{\mathcal{A}^c}(n)\mathds{1}_{\mathcal{A}^c}(n+1)\big)\Big)\\
&\leq \lim \limits_{k \to \infty} \frac{1}{m_k} \Big( 2 + \sum \limits_{n=0}^{m_{k}-2} \big(\mathds{1}_{\mathcal{A}}(n) + \mathds{1}_{\mathcal{A}}(n+1)\big) \Big) = 0.
\end{align*}

\end{proof}

Let us proceed now with the proof of Theorem \ref{thm:natural_counterexample}.

\begin{proof}[Proof of Theorem~{\rm\ref{thm:natural_counterexample}}]
By the construction of $f$, the set $\Lambda$ is a compact $T$-invariant set, and for every $(z,t) \in (\mathbb S^2 \setminus \{(0,0), \infty\}) \times \mathbb{S}^1$, we have $\dist(T^n(z,t), \Lambda)$ as $n \to \infty$. Hence, $\Lambda$ is an attractor for $T$ with the basin $B(\Lambda) = (\mathbb S^2 \setminus\{(0,0), \infty\}) \times \mathbb{S}^1$. To prove that $\mu$ is a natural measure for $T$, we show that the sequence of measures
	\[
	\mu_m = \frac{1}{m}\sum\limits_{n=0}^{m-1} \delta_{T^n (z,t)}
	\]
converges to $\mu$ in the weak-$^*$ topology for every $(z,t) \in (\mathbb S^2 \setminus (S \cup \{(0,0), \infty\}) \times \mathbb{S}^1$. It is enough to prove that every limit point of the sequence $\mu_m$ is equal to $\mu$. It follows from Corollary~\ref{cor:N_i} that every such limit point must be of the form $\nu_1/2 + \nu_2/2$, where $\nu_1$ is a probability measure on the circle $\{p\} \times \mathbb{S}^1$ and $\nu_2$ is a probability measure on the circle $\{q\} \times \mathbb{S}^1$. Our goal is to show that $\nu_1 = \delta_{(p, 0)}$ and $\nu_2 = \Leb_{\mathbb S^1}$, where $\Leb_{\mathbb S^1}$ is the Lebesgue measure on $\{q\} \times \mathbb S^1$. 

Take $m_k \nearrow \infty$ such that $\lim \limits_{k \to \infty} \mu_{m_k} = \nu_1/2 + \nu_2/2$. Let
	\[ \vartheta_{p, k} = \frac{1}{m_k}\sum\limits_{n=0}^{m_k - 1} \mathds{1}_{U_p}(f^n(z)) \,\delta_{T^n (z,t)},\qquad  \vartheta_{q, k} = \frac{1}{m_k}\sum\limits_{n=0}^{m_k - 1} \mathds{1}_{U_q}(f^n(z))\, \delta_{T^n (z,t)}\]
	and
	\[ \vartheta_{O, k} = \frac{1}{m_k}\sum\limits_{n=0}^{m_k - 1} \mathds{1}_{\mathbb S^2 \setminus (S \cup \{(0,0), \infty\} \cup U_p \cup U_q)}(f^n(z)) \, \delta_{T^n (z,t)}. \]
	Clearly,
	\[ \mu_{m_k} = \vartheta_{p, k} + \vartheta_{q, k} + \vartheta_{O, k}. \]
	By Corollary~\ref{cor:N_i},
	\[\lim \limits_{k \to \infty} \vartheta_{p, k} = \frac 1 2 \nu_1,\ \lim \limits_{k \to \infty} \vartheta_{q, k} = \frac 1 2 \nu_2\ \text{ and }\ \lim \limits_{k \to \infty} \vartheta_{O, k} = 0.\]
	Let 
	\[\pi \colon X \to \mathbb{S}^1, \qquad \pi(z,t) = t
	\]
be the projection. As $\supp\nu_1 \subset \{ p \} \times \mathbb{S}^1$ and $\supp\nu_2 \subset \{q\} \times \mathbb{S}^1$ and $g$, $R_\alpha$ are uniquely ergodic with invariant measures $\delta_0$ and $\Leb_{\mathbb S^1}$, respectively, it is enough to show that the limits of projected measures $\pi_*\vartheta_{p, k}$ and $\pi_*\vartheta_{q, k}$ are, respectively, $g$ and $R_\alpha$-invariant. 

	We have
	\[  \pi_*\vartheta_{p, k} =  \frac{1}{m_k}\sum\limits_{n=0}^{m_k - 1} \mathds{1}_{U_p}(f^n(z)) \, \delta_{\pi(T^n(z,t))}, \]
	Let 
	\[
	M_k = \sum\limits_{n=0}^{m_k - 1} \mathds{1}_{U_p}(f^n(z))
	\]
be the number of iterates $f^n(z)$ which are in $U_p$ up to time $m_k -1$ and let $(z_0, t_0), (z_1, t_1), \ldots$ be consecutive elements of the trajectory $\{T^n(z,t)\}_{n=0}^{\infty}$, such that $(z_j, t_j) \in U_p \times \mathbb{S}^1$. Then
\[ \pi_*\vartheta_{p, k} =  \frac{1}{m_k}\sum\limits_{j=0}^{M_k-1} \delta_{t_j}. \]
	Note that if $f(z_j) \in U_p$, then $t_{j+1} = g(t_j)$, so $\delta_{t_{j+1}} = g_* \delta_{t_j}$. Let $\mathcal A = \{ j \in \N : f(z_j) \notin U_{p} \}$. By Proposition~\ref{prop:N_i}, the set $\mathcal A$ has asymptotic density zero, as the time spent in $U_p$ by the trajectory of $z$ under $f$ during its $i$-th visit grows linearly with $i$, while during each visit only the last iterate is such that $f(z_j) \notin U_p$. We can therefore apply Lemma~\ref{lem:density_zero_limit} to conclude that the sequence $\frac{1}{M_k}\sum_{j=0}^{M_k-1} \delta_{t_j}$
	converges to a $g$-invariant probability measure, hence
	\[\lim \limits_{k \to \infty} \frac{1}{M_k}\sum\limits_{j=0}^{M_k-1} \delta_{t_j} = \delta_{0}.\]
	On the other hand, Corollary~\ref{cor:N_i} implies $\lim_{k \to \infty} \frac{M_k}{m_k} = \frac 1 2$, so
	\[\lim \limits_{k \to \infty} \pi_*\vartheta_{p, k} = \frac 1 2 \delta_0.\]
	By the same arguments we show
	\[\lim \limits_{k \to \infty} \pi_*\vartheta_{q, k} = \frac 1 2 \Leb_{\mathbb{S}^1}.\]
	Therefore, $\mu_m$ converges to $\mu$ in the weak-$^*$ topology and $\mu$ is a natural measure for $T$.

\end{proof}

\begin{rem}
To obtain a counterexample to the SSOY predictability conjecture in its original formulation, one can also perform a similar construction on a manifold with boundary $\B \times \mathbb{S}^1$, where $\B$ is a closed $2$-dimensional disc. Namely, it is enough to replace the diffeomorphism $f$ of $\mathbb{S}^2$ constructed in Subsection~\ref{subsec:S^2} with a diffeomorphism of $\B$, which is a suitable modification if the `Bowen's eye' example described e.g.~in \cite[Example 5.2.(B)]{Catsigeras14}, with properties similar to $f$.
\end{rem}

\bibliographystyle{alpha}
\bibliography{universal_bib}

\end{document}